%% file: main.tex
\documentclass[9pt]{article}

\usepackage{hyperref}
\usepackage{url}
\usepackage{smile}
\usepackage{graphicx}
\usepackage{algorithm}
\usepackage{algorithmic}
\usepackage{epstopdf}
\usepackage[margin=1.0in]{geometry}
\usepackage[export]{adjustbox}
\usepackage{mathtools, cuted}
\usepackage{bbm}
\usepackage{wrapfig}
\usepackage{subcaption}
\usepackage{caption}
\usepackage{enumitem}
\usepackage{tabularx}
\usepackage{tcolorbox}
\usepackage[nocompress]{cite}

\numberwithin{equation}{section}

\hypersetup{
    colorlinks=true,
    linkcolor=blue,
    anchorcolor=blue, 
    citecolor=blue,
}

\allowdisplaybreaks

\renewcommand{\frac}[2]{\tfrac{#1}{#2}}

\title{Sample Complexity for Markov Decision Processes and Stochastic Optimal Control with Static Risk Measures}

\author{
Cristian Chávez\thanks{
Pontificia Universidad Católica de Chile. (E-mail: \url{cristian.chvez@uc.cl}).
The majority of this work was completed during Cristian’s visit to Texas A\&M University.
}
\and 
Yan Li\thanks{Department of Industrial and Systems Engineering,
Texas A\&M University. (E-mail: \url{yan.li@tamu.edu}).}
}

\date{\vspace{-0.3in}}

\begin{document}
{
\makeatletter
\addtocounter{footnote}{2} 
\renewcommand\thefootnote{\@fnsymbol\c@footnote}%
\makeatother
\maketitle
}

\begin{abstract}
We present an elementary state augmentation method for a class of static risk measure applied to the total cost for both Markov decision processes (MDPs) and stochastic optimal control (SOC), such that dynamic programming equations can be derived on the augmented space. 
Through this we discuss the sample complexities of these two problem classes. 
We demonstrate the application of the proposed approach by developing a general framework for studying risk-averse MDPs and SOCs with distributionally robust functional generated by $\phi$-divergences, and obtain new sample complexity results for commonly used divergence functions.
\end{abstract}

\input{intro}

\input{risk_finite_mdp}
\input{risk_finite_soc}
\input{applications}

\input{conclusion}

\bibliographystyle{plain}
\bibliography{reference}

\end{document}

%% file: intro.tex

\section{Introduction}

We consider the following risk-averse sequential decision-making problem
\begin{align}\label{general_finite_horizon}
    \min_{\pi \in \Pi} \cR^\pi(s) \coloneqq
    \cR_{\Theta, f} \rbr{ \tsum_{t=0}^T c_t(S_t, A_t)}.
\end{align}
Here the risk functional $ \cR_{\Theta, f} $ is defined through 
\begin{align}\label{eq_risk_functional}
\textstyle
   \cR_{\Theta, f} (X)  = \min_{\theta \in \Theta} \EE \sbr{
   f_\theta(X) 
   },
\end{align}
for some parameterized nonlinear function $f_\theta: \RR \to \RR$ with parameter space $\Theta \subseteq \RR^{d_\Theta}$.
At every stage $t$, the decision maker takes an action $A_t$ from the action space $\cA_t$ based on the current state $S_t \in \cS_t$,
which subsequently incurs a cost of $c_t(S_t, A_t) \in [0,1]$. 
Depending on how the distribution of $S_{t+1}$ is modeled given $S_t$ and $A_t$, 
we will consider problem \eqref{general_finite_horizon} 
in the context of Markov decision process (MDP) and stochastic optimal control (SOC).
Unless stated otherwise, we consider the set of randomized history-dependent policy class $\Pi$, i.e., $\pi \in \Pi: \cH_t \to \Delta_{\cA_t}$, where $\Delta_{\cA_t}$ denotes the probability distributions supported over the action space,
and $H_t \in \cH_t$ denotes the history up to stage $t$.
The above expectation is then taken with respect to the probability law of the data process $\cbr{(S_t, A_t)}_{t \geq 0}$ where $S_0 = s$, $A_t \sim \pi(\cdot|H_t)$ under policy $\pi$.

Below we discuss some examples of risk measures that can be represented in the form of \eqref{eq_risk_functional}.

\begin{example}[Conditional value-at-risk \cite{rockafellar2000optimization}]\label{example_cvar}
    Let $X$ be a random variable on $(\Omega, \cF, P)$. Its conditional value-at-risk with risk level $\alpha \in (0,1)$ is given by 
    \begin{align*}
     \cR(X) = \min_{\theta}  \cbr{ \theta + \frac{1}{\alpha} \EE_P \sbr{X - \theta}_+ }.
    \end{align*}
    In this case we can take
$
        f_\theta(z) = \theta + \frac{1}{\alpha} [z - \theta]_+,
$
    and $\Theta \subseteq \RR$ in \eqref{eq_risk_functional}.
\end{example}

The conditional value-at-risk can be viewed as a special case of the distributionally robust functional generated by the following  $\phi$-divergence-based ambiguity sets \cite{shapiro2017distributionally}.

\begin{example}[$\phi$-divergence distributionally robust functional]\label{example_phi_divergence_dro}
Let $\phi: \RR \to \RR_+ \cup \cbr{+\infty}$ be a convex lower-semicontinuous function with $\phi(1)= 0$ and $\phi(x) = +\infty$ for $x < 0$.
Let $X$ be a random variable on $(\Omega, \cF, P)$.
For $\tau\geq 0$, consider the set of density functions 
\begin{align*}
    \mathfrak{M}_\tau = \cbr{
    \zeta \in \mathfrak{D}: \int_{\Omega} \phi(\zeta(w)) d P(w) \leq \tau
    },
\end{align*}
where $\mathfrak{D}$ denotes the set of probability density functions with respect to $P$.
We refer to $\mathfrak{M}_\tau$ as the ambiguity set. 
The distributionally robust functional $\cR(X)$ is then defined  as
\begin{align}\label{dro_phi}
    \cR(X) = \sup_{\zeta \in \mathfrak{D}} \int_{\Omega} X(\omega) \zeta(\omega) d P(\omega) , ~ \mathrm{s.t.}~ 
    \zeta \in \mathfrak{M}_\tau.
\end{align}
That is, we measure the risk of $X$ by taking its worst-case expectation, when the distribution defining the expectation is allowed to have an arbitrarily chosen density function inside the ambiguity set $\mathfrak{M}_\tau$.
From \cite{shapiro2017distributionally}, \eqref{dro_phi} admits the following dual problem 
\begin{align}\label{eq_phi_dro_f_form}
    \cR(X) = \min_{\lambda \geq 0, \mu \in \RR}
    \cbr{ \lambda \tau + \mu + 
    \EE_P \sbr{ (\lambda \phi)^* (X - \mu) } },
\end{align}
where $\lambda, \mu$ correspond to dual variables associated with constraints in \eqref{dro_phi},
and $(\lambda \phi)^*$ denotes the conjugate of $\lambda \phi$.
In this case, one can take $\theta = (\lambda, \mu)$,  
\begin{align*}
    f_{(\lambda, \mu)}(z) = \lambda \tau + \mu + (\lambda \phi)^* (z - \mu),
\end{align*}
and accordingly take $\Theta = \cbr{(\lambda, \mu): \lambda \geq 0, \mu \in \RR}$. 
\end{example}


We refer to risk measures of the form \eqref{general_finite_horizon} as the static risk measures, since the decision maker considers the risk applied to the total cost associated with the policy. 
In contrast to the risk-neutral setting where the risk corresponds to the expectation, for the static risk measures, a central challenge is the lack of dynamic programming (DP) equations except for certain special cases \cite{fei2020risk}.
To address this challenge, DP equations defined with the assistance of augmented state variables have been introduced for risk-averse MDPs \cite{bauerle2011markov,chow2015risk,hau2023dynamic,bauerle2014more}. 
In particular, for conditional value-at-risk (Example \ref{example_cvar}), \cite{bauerle2011markov} proposes DP equations based on its variational representation and subsequently defines the augmented state variable as the accumulated total cost,
and \cite{chow2015risk} proposes DP equations based on the dual representation and defines the augmented state variable as the risk level.\footnote{It has been recently pointed out in \cite{hau2023dynamic} that the DP equations in \cite{chow2015risk} only hold for policy evaluation.} 
The approach in \cite{bauerle2011markov} has since been generalized to general nonlinear utility functions \cite{bauerle2014more,wu2023risk}, optimized certainty equivalent \cite{wang2024reductions}, and a class of spectral risk measures \cite{bastani2022regret}.
With DP equations in place, there has been an active line of research in developing computational methods for risk-averse MDPs \cite{chow2015risk,ding2022sequential,chow2014algorithms}, and in determining the sample complexities of finding an approximate optimal control policy when the underlying model is unknown and has to be learned from the data \cite{wang2023near,wu2023risk,wang2024reductions,ni2024risk,bastani2022regret}. 

It should also be noted that a related line of research considers nested risk measures, where risk is applied iteratively at every stage of the decision process, and hence DP equations can be naturally derived. 
We refer to \cite{ruszczynski2010risk,li2025rectangularity} for relevant  discussions on MDPs and to \cite{shapiro2025risk} for SOCs.
The nested risk measure for MDPs is also closely related to robust MDPs with $(\mathrm{s,a})$-rectangular ambiguity sets \cite{iyengar2005robust,nilim2005robust} through the dual representation of coherent risk measures \cite{shapiro2021lectures}.
Correspondingly, computational methods and sample complexities have been studied for a variety of ambiguity sets \cite{panaganti2022sample,wang2023finite,wang2022policy,wang2023policy,li2022first}.

 \vspace{0.05in}

{\bf Contributions.}
In this manuscript, we study the DP equations and the sample complexities associated with static risk-averse MDPs and SOCs \eqref{general_finite_horizon}, when the risk can be represented in the form of \eqref{eq_risk_functional}.
Our contributions can be summarized as follows.

First, we propose an elementary construction of DP equations for \eqref{general_finite_horizon}, based on augmenting the state variable by the accumulated total cost. The proposed DP equations can be applied to both the MDP and SOC models without essential changes.
In contrast to \cite{bauerle2011markov}, the underlying idea behind our construction is that the nonlinearity of $f_\theta$ in \eqref{eq_risk_functional} can be handled through simple telescoping operations, which produces dense cost functions.
Through this, one can, in a straightforward manner, define an MDP (resp. SOC) over the augmented state space with a modified cost function. 
 A similar cost function has also been recently studied in the context of conditional value-at-risk \cite{muni2026reward} for infinite-horizon MDPs.
 Indeed, our approach can also be naturally extended to the infinite-horizon setting, while removing the need to restrict to a specialized function class in \cite{bauerle2014more} to ensure the existence of solutions for the DP equations.

Second, using the proposed DP equations, we study  the sample complexities for \eqref{general_finite_horizon} for both the MDP and  the SOC models. 
For risk-averse policy evaluation, where one aims to evaluate the risk $\cR^\pi$ defined in \eqref{general_finite_horizon} for a given policy, we establish a sample complexity of 
 $\tilde{\cO}(\abs{\cS}\abs{\cA} T^3 L_T^2/\epsilon^2)$ (resp. $\tilde{\cO}(T^4 L_T^2/\epsilon^2)$) for MDPs (resp. SOCs),
 where $T$ denotes the number of stages, and $L_T$ is the Lipschitz constant for the nonlinear function $f_\theta$ in \eqref{eq_risk_functional} that potentially  depends on $T$ (see Assumption \ref{assump_basic}).
 For risk-averse policy optimization \eqref{general_finite_horizon}, we establish a sample complexity of 
 $\tilde{\cO}(\abs{\cS}\abs{\cA} T^4 L_T^2/\epsilon^2)$ (resp. $\tilde{\cO}(T^5 L_T^2/\epsilon^2)$) for MDPs (resp. SOCs).
 In the context of risk-neutral MDPs, the obtained results recover the optimal sample complexity for finite-horizon time-inhomogeneous MDPs \cite{li2020breaking}.
 In addition, in the case of $\Theta$ being a singleton, the static risk functional in \eqref{eq_risk_functional} we consider reduces to the nonlinear utility function studied in \cite{bauerle2014more,wu2023risk}, and in this case improves upon the results for risk-averse MDPs in \cite{wu2023risk} by a factor of $\cO(T^3)$.
 The obtained results also appear to be the first sample complexity for risk-averse SOCs with static risk measures.

 Third, we demonstrate the application of the aforementioned approach to study the sample complexity of risk-averse MDPs and SOCs, when the risk $\cR^\pi$ corresponds to the distributionally robust functional generated by $\phi$-divergences. 
In particular, we establish results applicable to general $\phi$-divergences whose tail growth is superlinear, and demonstrate that \eqref{general_finite_horizon} becomes statistically intractable in the absence of the superlinear growth. 
It should be noted that in view of \eqref{eq_phi_dro_f_form} in Example \ref{example_phi_divergence_dro}, the function $f_\theta$ corresponding to $\phi$-divergences typically lacks the Lipschitz property (i.e., $L_T =\infty$), leading to a vacuous sample complexity bound under a naive application of the results  developed in the general context. 
  We develop novel techniques that resolve this issue via domain restriction,  and subsequently provide explicit characterizations of the impact of divergence function on the sample complexity. 
Interestingly, we show that for a general $\phi$-divergence, the sample complexity depends on $T$ and $\epsilon$ via the growth function associated with $\phi$,
and may not necessarily scale as $\cO(T^3/\epsilon^2)$ (resp. $\cO(T^4/\epsilon^2)$) as in the risk-neutral setting even for policy evaluation.
We then instantiate the obtained sample complexity for various concrete $\phi$-divergences, including conditional value-at-risk. 
In addition, we establish for the first time the sample complexity of risk-averse MDPs and SOCs associated with 
Kullback-Leibler divergence, $\chi^2$-divergence, and the general Cressie-Read class of divergences. 
Collectively, the obtained results demonstrate that the sample complexity for risk-averse policy optimization may differ from that of the risk-neutral setting, depending on the properties of the considered risk functional.

 \vspace{0.05in}

{\bf Organization.}
The rest of the manuscript is organized as follows. 
 Section \ref{sec_risk_mdp} discusses the DP equations for MDPs and subsequently develops sample complexity results for general risk-averse MDPs. 
Section \ref{sec_soc} then establishes the counterparts of these findings for the SOC model. 
Finally, we develop in Section \ref{sec_application} the concrete applications of Sections \ref{sec_risk_mdp} and \ref{sec_soc} in the context of distributionally robust risk functionals associated with $\phi$-divergences, 
and report new sample complexity results for commonly used $\phi$-divergences.

 \vspace{0.05in}

{\bf Notations.}
Unless stated otherwise, we reserve $\norm{\cdot}$ for the Euclidean distance. 
For any set $A$, we use $\delta_{A}$ to denote its indicator function, i.e., $\delta_{A}(x) = 0$ for $x \in A$ and $+\infty$ otherwise. 

%% file: risk_finite_mdp.tex


\section{Risk-averse Markov Decision Processes}\label{sec_risk_mdp}
In this section, we will proceed to discuss risk-averse MDP with finite-horizon setting.
For any MDP instance $\cM$, we use $\Pi(\cM)$ to denote the set of randomized history-dependent policies,
and $\Pi_{\mathrm{MR}}(\cM)$ for the set of randomized Markovian policies.

\subsection{Finite-horizon Risk-averse Markov Decision Processes}\label{sec_finite_mdp}
Consider a finite-horizon MDP 
$\cM = (
\cbr{\cS_t}_{t=0}^T, \cbr{\cA_t}_{t=0}^T, \cbr{\PP_t}_{t=0}^{T-1},
\cbr{c_t}_{t=0}^T
)
$.
Here $\cS_t$ denotes the finite state space at stage $t$,  $\cA_t$ denotes its corresponding finite action space,
 $c_t: \cS_t \times \cA_t \to [0,1]$ 
denotes the cost function, and
$\PP_t: \cS_t \times \cA_t \to \Delta_{\cS_{t+1}}$
denotes the transition kernel that maps a given state-action pair onto the set of  probability distributions $\Delta_{\cS_{t+1}}$ supported over $\cS_{t+1}$.
For notational convenience, we will occasionally write $\PP$ as shorthand for $\cbr{\PP_t}_{t=0}^{T-1}$.
We seek to find the optimal policy~of the following risk,
\begin{align}\label{eq_risk_finite_horizon_mdp}
    \min_{\pi \in \Pi(\cM)} \cR^\pi(s) \coloneqq
    \cR_{\Theta, f} \rbr{ \tsum_{t=0}^T c_t(S_t, A_t)}
\end{align}
where
the data process $\cbr{(S_t, A_t)}_{t=0}^T$ is generated by $S_0 = s$, $A_t \sim \pi(\cdot|H_t)$,  $S_{t+1} \sim \PP_t(\cdot|S_t, A_t)$,
and $H_t = (S_0, A_0, \ldots, S_t)$ denotes the history up to stage $t$.
In view of the definition of $\cR_{\Theta, f}$, we have 
\begin{align}\label{value_given_theta_finite_horizon_mdp}
    \cR^\pi (s) = \min_{\theta \in \Theta} \cbr{ V_0^{\pi, \theta}(s)
    \coloneqq \EE^\pi_\PP \sbr{f_\theta \rbr{\tsum_{t=0}^T c_t(S_t, A_t)} | S_0 = s} }.
\end{align}
It is then clear that  \eqref{eq_risk_finite_horizon_mdp} is equivalent to 
\begin{align}\label{finite_horizon_minmin_mdp}
    \min_{\theta \in \Theta} \min_{\pi \in \Pi(\cM)} V_0^{\pi, \theta}(s).
\end{align}

A potential challenge in solving \eqref{eq_risk_finite_horizon_mdp} is the nonlinearity of $f_\theta$ in \eqref{value_given_theta_finite_horizon_mdp}, which breaks the dynamic decomposition of \eqref{value_given_theta_finite_horizon_mdp} that underlies the dynamic programming equations associated with risk-neutral MDPs.
Our first observation in this section is the following  simple reformulation that allows us to define the dynamic equations over an augmented state space.

\begin{definition}[Finite-horizon augmented MDP]\label{def_aug_fin_mdpite_horizon}
    For any $\theta \in \Theta$, the augmented MDP $\tilde{\cM}_\theta$ is defined as follows. 
    The state space is $\tilde{\cS}_t = \cS_t \times \cX_t$, where $\cX_t = [0,t]$, and the action space is $\tilde{\cA}_t = \cA_t$.
    The cost function and the transition kernel are defined by 
\begin{align}\label{cost_transition_aug_fin_mdpite}
\tilde{c}_t^\theta(s,x,a)
&= f_\theta ({x+c_t(s,a)})- f_\theta(x), \nonumber \\
\tilde{\PP}_t(s',x' |  s,x,a)
&= \PP_t(s' |  s,a) \mathbbm{1}_{\cbr{x'=x+c_t(s,a)}}.
\end{align}
\end{definition}

In view of \eqref{cost_transition_aug_fin_mdpite} in Definition \ref{def_aug_fin_mdpite_horizon}, the augmented MDP $\tilde{M}_\theta$ tracks the accumulated total cost through the augmented state variable $X_t \in \cX_t$ at every stage $t$.
Given this, it is not difficult to see that the  set of history-dependent, randomized policies for $\tilde{\cM}_\theta$ is also $\Pi(\cM)$.
Then for any policy $\pi \in \Pi(\cM)$ and any $\theta \in \Theta$, we define the value function  as
\begin{align}\label{eq_value_given_pi_aug_mdp}
    \tilde{V}^{\pi, \theta}_0(s, x) = 
    \EE^{\pi}_{\tilde{\PP}} \sbr{
    \tsum_{t=0}^T \tilde{c}_t^\theta(S_t, X_t, A_t)
    | S_0 = s, X_0 = x
    }.
\end{align}

With Definition \ref{def_aug_fin_mdpite_horizon} in place, we make the following  observation. 

\begin{proposition}\label{prop_reform_fin_mdpite}
    We have for any $\pi \in \Pi(\cM)$,
    \begin{align*}
         V^{\pi,\theta}_0(s) =  \tilde{V}^{\pi, \theta}_0(s, 0) + f_\theta(0), ~ \forall s \in \cS_0.
    \end{align*}
\end{proposition}

\begin{proof}
    It is immediate to verify that 
    for any policy $\pi \in \Pi(\cM)$,
    \begin{align*}
         \tilde{V}^{{\pi}, \theta}_0(s, 0) 
         & =
         \EE^{{\pi}}_{\tilde{\PP}} \sbr{
    \tsum_{t=0}^T \tilde{c}_t^\theta(S_t, X_t, A_t) | S_0 = s, X_0 = 0
    }  \\
    & = 
     \EE^{{\pi}}_{\tilde{\PP}} \sbr{
    \tsum_{t=0}^T 
    f_\theta ({X_t+c_t(S_t,A_t)})-f_\theta(X_t) | S_0 = s, X_0 = 0
    } \\
    & = \EE^{{\pi}}_{{\PP}}
    \sbr{f_\theta \rbr{\tsum_{t=0}^T c_t (S_t, A_t)} | S_0 = s}
    - f_\theta(0)
         = V^{\pi,\theta}_0(s) - f_\theta(0),
    \end{align*}
where the last equality follows from the definition of $\tilde{c}_t^\theta$ and $\tilde{\PP}_t$ in \eqref{cost_transition_aug_fin_mdpite}. This completes the proof.
\end{proof}

In view of Proposition \ref{prop_reform_fin_mdpite}, one can immediately obtain the following dynamic programming equations for the inner minimization of  \eqref{finite_horizon_minmin_mdp}  through the augmented MDP $\tilde{\cM}_\theta$.

\begin{proposition}\label{thrm_dp_fix_theta_finite}
    Define the optimal cost-to-go function in the augmented MDP $\tilde{\cM}_\theta$ as 
    \begin{align*}
        \tilde{V}_t^\theta (s, x) = 
        \min_{\pi \in \Pi(\cM)}
        \EE^{\pi}_{\tilde{\PP}} 
        \sbr{
        \tsum_{h=t}^T \tilde{c_h}^\theta(S_h, X_h, A_h) | S_t = s, X_t = x
        }.
    \end{align*}
    Then we have 
    $\tilde{V}_T^\theta(s,x) = \min_{a \in \cA_T} \tilde{c}_T^\theta(s, x, a)$ for $(s,x) \in \cS_T \times \cX_T$,  and 
    \begin{align}\label{dp_val_give_theta_finite}
        \tilde{V}^\theta_t(s, x) = \min_{a \in \cA_t}
        \tilde{c}_t^\theta(s, x, a) + 
        \tsum_{s' \in \cS_{t+1}} \PP_t(s'|s, a) \tilde{V}_{t+1}^\theta(s', x + c_t(s, a)),
        ~ \forall  (s, x) \in \cS_t \times \cX_t,
    \end{align}
    for any $0 \leq t \leq T-1$.
    In addition, the policy ${\pi^\theta}$ defined as 
    \begin{align}\label{opt_pi_given_theta_finite}
        \pi^\theta_t(s, x) \in \Argmin_{a \in \cA_t} \cbr{
        \tilde{c}_t^\theta(s, x, a) + 
        \tsum_{s' \in \cS_{t+1}} \PP_t(s'|s, a) \tilde{V}_{t+1}^\theta(s', x + c_t(s, a))
        }
    \end{align}
    is an optimal policy for $\min_{\pi \in \Pi(\cM)} V^{\pi, \theta}_0(s)$.
\end{proposition}

\begin{proof}
    The first part of the claim \eqref{dp_val_give_theta_finite} is a direct consequence of dynamic equations applied to the augmented MDP $\tilde{\cM}_\theta$. 
    The rest of the claim follows from Proposition \ref{prop_reform_fin_mdpite}.
\end{proof}

Given Proposition \ref{thrm_dp_fix_theta_finite}, we can obtain the dynamic equations for \eqref{eq_risk_finite_horizon_mdp} as follows.

\begin{theorem}\label{cor_dp_optimal_value_finite_mdp}
Let $\theta^*$ be an optimal solution of \eqref{finite_horizon_minmin_mdp}.
Consider $\{\tilde{V}^{\theta^*}_t\}_{t=0}^T$ defined as in \eqref{dp_val_give_theta_finite}
 and $\pi^{\theta^*}$ defined in \eqref{opt_pi_given_theta_finite}. 
 Then we have 
 \begin{align*}
     \min_{\pi \in \Pi} \cR^\pi(s) = \tilde{V}^{\theta^*}_0(s, 0) + f_{\theta^*}(0), 
 \end{align*}
 and $\pi^{\theta^*}$ is an optimal policy.
\end{theorem}

In view of Theorem \ref{cor_dp_optimal_value_finite_mdp}, the optimal policy $\pi^*$ for \eqref{eq_risk_finite_horizon_mdp} is non-randomized and  depends on the history only through the accumulated cost. 
We now turn our attention to evaluating $\cR^\pi(s)$ for any given policy $\pi$. Following a similar argument as before, we proceed to establish the dynamic programming equations 
of $\cR^\pi(s)$ for any $\pi \in \Pi_{\mathrm{MR}}(\tilde{\cM}_\theta)$.

\begin{theorem}\label{theorem_risk_dp_fix_policy}
    For any $\theta \in \Theta$ and $\pi \in \Pi_{\mathrm{MR}}(\tilde{\cM}_\theta)$, define the value function of $\pi$ in $\tilde{M}_\theta$ as 
    \begin{align*}
        \tilde{V}_t^{\pi , \theta} (s, x) = 
        \EE^{\pi}_{\tilde{\PP}} 
        \sbr{
        \tsum_{h=t}^T \tilde{c_h}^\theta(S_h, X_h, A_h) | S_t = s, X_t = x
        }.
    \end{align*}
    Then we have 
    $\tilde{V}^{\pi, \theta}_T(s,x) = \tsum_{a \in \cA_T} \pi_T(a|s,x) \tilde{c}_T^\theta(s,x,a)$ for any $(s, x) \in \cS_T \times \cX_T$, and 
    \begin{align}\label{dp_finite_mdp_fix_policy}
        \tilde{V}^{\pi, \theta}_t(s, x) = \EE_{a \sim \pi(\cdot|s, x)} \sbr{
        \tilde{c}_t^\theta(s, x, a) + 
        \tsum_{s' \in \cS_{t+1}} \PP_t(s'|s, a) \tilde{V}_{t+1}^{\pi, \theta}(s', x + c_t(s, a))
        }, 
    \end{align}
    for any $0 \leq t \leq T-1$ and $ (s,x) \in \cS_t \times \cX_t$.
    In addition,  
denote $\theta_\pi \in \Argmin_{\theta \in \Theta} \cbr{\tilde{V}^{\pi, \theta}_0(s, 0) + f_\theta(0)}$.
We have for any $\pi \in \Pi_{\mathrm{MR}}(\tilde{M}_\theta)$,
    \begin{align*}
        \cR^\pi(s) = \tilde{V}^{ \pi, \theta_\pi}_0(s, 0) + f_{\theta^\pi}(0).
    \end{align*}
\end{theorem}
\begin{proof}
    The first part of the claim \eqref{dp_finite_mdp_fix_policy} follows from standard dynamic programming principles applied to the augmented MDP. 
    The rest of the claim follows from the definition of $\cR^\pi$, $\theta_\pi$ and Proposition \ref{prop_reform_fin_mdpite}.
\end{proof}


\subsection{Sample Complexity of Risk-averse Policy Evaluation and Optimization}\label{subsec_sample_finite_mdp}

We now turn our attention to \eqref{eq_risk_finite_horizon_mdp} when the underlying transition kernels $\cbr{\PP_t}$ are unknown. 
Instead, we assume sample access to ${\PP_t}$ such that for any $(s, a) \in \cS_t \times \cA_t$, one can generate $n$ i.i.d. samples from the distribution $\PP_t(\cdot|s, a)$.
To facilitate our discussion, let us denote by $\hat{\PP}_{t,n}$ the empirical estimate of $\PP_t$ constructed from such samples, 
 write $\hat{\PP}_n$ in short for $\{\hat{\PP}_{t,n}\}_{t=0}^{T-1}$,
and $\abs{\cS} = \max_{0 \leq t \leq T} \abs{\cS_t}$
together with $\abs{\cA} = \max_{0 \leq t \leq T} \abs{\cA_t}$.
We make the following assumption on the nonlinear function $f_\theta$ and the parameter space $\Theta$ throughout the rest of the manuscript. 

\begin{assumption}\label{assump_basic}
    We assume that $f_\theta$ is $(L_\Theta, L_C)$-Lipschitz, i.e.,
    \begin{align*}
        \abs{f_\theta(z) - f_\theta(z')}  \leq L_C \abs{z - z'}, ~~
        \abs{f_{\theta}(z) - f_{\theta'}(z)}  \leq L_\Theta \norm{\theta - \theta'}, 
        ~ \forall z, z' \in [0, T+1], ~~ \theta, \theta' \in \Theta.
    \end{align*}
    In addition, $\Theta$ is bounded in $\norm{\cdot}$ with radius $R_\Theta$.
\end{assumption}
It should be noted that our discussion allows $L_C$ and $L_\Theta$ to be potentially dependent on the number of stages $T$.
This will be of particular importance in Section \ref{sec_application} when we discuss general $\phi$-divergences.


\subsubsection{Risk-averse Policy Evaluation}\label{subsec_policy_eval_mdp}
For a given realization of $\hat{\PP}_n$ and any policy $\pi$, let us denote 
$\hat{\cR}^\pi_n$ as the risk  of policy $\pi$ defined as in \eqref{value_given_theta_finite_horizon_mdp}, with $\PP$ replaced by  $\hat{\PP}_n$ therein. 
For the rest of our discussion in this section, we are interested in evaluating the risk $\cR^\pi$  via the estimator $\hat{\cR}^\pi_n$.
That is, for  any target precision $\epsilon > 0$, we are interested in the number of samples $n$ necessary to ensure 
$\sup_{s \in \cS_0} |{ \hat{\cR}^\pi_n(s) - \cR^\pi(s)}| \leq \epsilon$ with high probability. 

Let us denote $\hat{V}_{t,n}^{\pi,\theta}$ as the value function of policy $\pi$ in the augmented MDP, defined in the same way as $\tilde{V}^{\pi, \theta}_t$ in \eqref{dp_finite_mdp_fix_policy}, with $\PP$ replaced by $\hat{\PP}_n$ therein.
We will restrict our focus in this section to a policy $\pi$ that is $L_\pi$-Lipschitz with respect to the augmented state variable, i.e., 
\begin{align}\label{eq_finite_mdp_policy_lip}
    \|\pi_t(\cdot|s,x) - \pi_t(\cdot|s,x')\|_1 \leq L_\pi\abs{x-x'}, ~ \forall s \in \cS_t, x,x' \in \cX_t,
\end{align}
for some $L_\pi > 0$.
We will discuss in Remark \ref{remark_lipschitz_policy} the potential relaxation of this restriction and the associated price in terms of the final sample complexity. 
As a direct consequence, we can establish the Lipschitz continuity of the value function for policy $\pi$ in the augmented MDP $\tilde{\cM}_\theta$.

\begin{lemma}\label{lemma_fin_mdp_pol_eval_val_lip}
    For any $0 \leq t \leq T$ and $s \in \cS_t$, the value functions $\tilde{V}_t^{\pi, \theta}(s, \cdot)$
    and $\hat{V}_{t,n}^{\pi, \theta}(s, \cdot)$ are Lipschitz with modulus   
    \begin{align*}
        L^\pi_{t} \coloneqq (T-t+1)L_C\left(2 + \tfrac{(T-t+2)L_\pi}{2}\right).
    \end{align*}
\end{lemma}

\begin{proof} 
    We proceed with an inductive argument. For $t = T$, we have
    \begin{align*}
        & |{\tilde V_T^{\pi, \theta}(s,x) - \tilde V_T^{\pi, \theta}(s,y)} | \\
        \leq & |\EE_{a \sim \pi_T(\cdot|s,x)}[\tilde{c}_T^\theta(s,y,a)] - \EE_{a \sim \pi_T(\cdot|s,y)}[\tilde{c}_T^\theta(s,y,a)]| + |\EE_{a \sim \pi_T(\cdot|s,x)}[\tilde{c}_T^\theta(s,x,a) - \tilde{c}_T^\theta(s,y,a)]| \\
        \overset{(a)}{\leq} & \max_{a \in \cA_T} |\tilde{c}_T^\theta(s,y,a)| \cdot \|\pi_T(\cdot|s,x) - \pi_T(\cdot|s,y)\|_1 + 2L_C|x-y| \\
        \leq & L_C(2 + L_\pi)\abs{x-y},
    \end{align*}
    where $(a)$ follows from Hölder's inequality and Assumption \ref{assump_basic},
    which implies $\norm{\tilde{c}_t}_\infty \leq L_C$ for $0 \leq t \leq T$.
    In addition, it is also clear that 
    $\norm{\tilde{V}^{\pi, \theta}_t}_\infty \leq  (T-t+1) L_C$.
        Consequently, suppose the desired claim holds at $t+1$, then we have
    \begin{align*}
        & |\tilde V_t^{\pi, \theta}(s,x) - \tilde V_t^{\pi, \theta}(s,y)| \\
        \leq & | \EE_{a \sim \pi_t(\cdot|s,x)}[ \tilde{c}_t^\theta(s,y,a) + \EE_{s' \sim \PP_t(\cdot|s,a)}[ \tilde V_{t+1}^{\pi, \theta}(s', y + c_t(s,a))]] \\
        & ~ ~ ~ ~ - \EE_{a \sim \pi_t(\cdot|s,y)}[\tilde{c}_t^\theta(s,y,a) + \EE_{s' \sim \PP_t(\cdot|s,a)}[ \tilde V_{t+1}^{\pi, \theta}(s', y+ c_t(s,a))]]| \\
        & + |\EE_{a \sim \pi_t(\cdot|s,x)}[\tilde{c}_t^\theta(s,x,a) - \tilde{c}_t^{\theta} (s,y,a) + \EE_{s' \sim \PP_t(\cdot|s,a)}[\tilde V_{t+1}^{\pi, \theta}(s', x + c_t(s,a)) - \tilde V_{t+1}^{\pi, \theta}(s', y + c_t(s,a))]]| \\
         \leq &  
        (T-t+1) L_C L_\pi \abs{x-y}
        + 2L_C \abs{x-y} + L^\pi_{t+1} \abs{x-y},
    \end{align*}
    where the last inequality follows from Hölder's inequality, Assumption \ref{assump_basic}, and the inductive hypothesis. 
\end{proof}

To set up our ensuing discussion on the estimation error $|{ \hat{\cR}^\pi_n(s) - \cR^\pi(s)}|$, recall from Theorem \ref{theorem_risk_dp_fix_policy} that 
\begin{align*}
    \cR^\pi(s) = \min_{\theta \in \Theta} \tilde{V}^{\pi, \theta}_0(s,0) + f_\theta(0),  ~~~
        \hat{\cR}_n^\pi(s) = \min_{\theta \in \Theta} \hat{V}^{\pi, \theta}_{0,n}(s,0) + f_\theta(0). 
\end{align*}
Hence it suffices to control 
$(\hat{V}_{0,n}^{\pi, \theta} - \tilde{V}_0^{\pi, \theta})(s,0)$
uniformly over $\theta \in \Theta$.
To this end, we establish below an error decomposition that decomposes 
 $(\hat{V}_{t,n}^{\pi, \theta} - \tilde{V}_t^{\pi, \theta})(s,x)$ into a sum of martingale differences, and subsequently applies Freedman's inequality.

\begin{lemma}\label{lemma_err_bernstein_eval}
\label{lemma_fin_mdp_pol_eval_berns_sum}
Fix $0 \leq t \leq T-1$,  $s \in \cS_t$ and $ x \in \cX_t$. Then for any $\delta \in (0,1)$, with probability at least $1-\delta$, we have
    \begin{align}
    \abs{(\hat{V}_{t,n}^{\pi, \theta} - \tilde{V}_t^{\pi, \theta})(s,x)} \leq \sqrt{\tfrac{2 }{n}\hat{W}_{t,n}^{\pi, \theta}(s,x) \log(\tfrac{2}{\delta})} + \tfrac{2  T L_C}{3n} \log(\tfrac{2}{\delta}), \nonumber
    \end{align}
    where
    \begin{align} 
    \hat{W}_{t,n}^{\pi, \theta}(s,x) \coloneqq  \tsum_{h = t}^{T-1} \EE_{\tilde \PP}^\pi \left[ \left. \Var_{\tilde{\PP}_h^\pi (\cdot|S_{h}, X_{h})} (\hat{V}_{h+1,n}^{\pi, \theta}) \right| (S_t, X_t) = (s,x) \right], 
    \label{fin_mdp_pol_eval_err_var_sum_decom}
    \end{align}
and $    \tilde{\PP}_t^\pi(s', x'|s,x) = \EE_{a \sim \pi_t(\cdot|s,x)} \tilde{\PP}_t( s', x'  | s, x, a)$.
\end{lemma}
\begin{proof}
    From \eqref{dp_finite_mdp_fix_policy}, we obtain
    \begin{align}
        (\hat{V}_{t,n}^{\pi, \theta} - \tilde{V}_t^{\pi, \theta})(s,x) & = (\EE_{\hat{\PP}_{t,n}^\pi(\cdot|s,x)} - \EE_{\tilde{\PP}_t^\pi(\cdot|s,x)})[\hat{V}_{t+1,n}^{\pi, \theta}] + \EE_{\tilde{\PP}_t^\pi(\cdot|s,x)}[\hat{V}_{t+1,n}^{\pi, \theta} - \tilde{V}_{t+1}^{\pi, \theta}] \label{fin_mdpite_pol_eval_bell_err_rec}\\
        & = \tsum_{h=t}^{T-1} \EE_{\tilde \PP}^\pi \left[\left.(\EE_{\hat{\PP}_{h,n}^\pi(\cdot|S_h,X_h)} - \EE_{\tilde{\PP}_h^\pi(\cdot|S_h,X_h)})[\hat{V}_{h+1,n}^{\pi, \theta}]\right| (S_t, X_t) = (s,x)\right], \nonumber
    \end{align}
    where the last equality follows from recursively applying \eqref{fin_mdpite_pol_eval_bell_err_rec},
    and $\hat{\PP}_{h,n}$ denotes the empirical estimate of $\tilde{\PP}_h$.
    For any $h \geq t$,   $(s,x) \in \cS_t \times \cX_t$,
    and $(s', x') \in \cS_h \times \cX_h$,
    let   $\mu_{t,h}^\pi (s',x'|s,x)$ denote the probability of reaching $(s',x')$ at stage $h$ under policy $\pi$ starting  from state $(s,x)$ at stage $t$. 
    Applying \eqref{dp_finite_mdp_fix_policy} recursively, we have
    \begin{align}
        (\hat{V}_{t,n}^{\pi, \theta} - \tilde{V}_t^{\pi, \theta})(s,x) & = \tsum_{h=t}^{T-1} \EE_{(s',x') \sim \mu_{t,h}^{\pi}(\cdot|s,x)} \left[(\EE_{\hat{\PP}_{h,n}^\pi(\cdot|s',x')} - \EE_{\tilde{\PP}_h^\pi(\cdot|s',x')})[\hat{V}_{h+1,n}^{\pi, \theta}] \right]. \label{fin_mdp_pol_eval_err_sum_decom}  
    \end{align}    
For any $i \in [n]$, let $Y_{\{i,h,s,a\}}$ be the $i$-th sample from $\PP_h(\cdot\given s,a)$ used to construct $\hat{\PP}_{h,n}(\cdot\given s,a)$. 
Let us define 
$
\delta_{i, h}(s', x', a') = \hat V_{h+1,n}^{\pi,\theta}
    \rbr{Y_{\{i,h,s',a'\}},x'+c_h(s',a')}
    -
    \EE_{\tilde{\PP}_h^\pi(\cdot\given s',x')}
    [{\hat V_{h+1,n}^{\pi,\theta}}],
$
and
\begin{align*}
    \Delta_{i,h}
    \coloneqq
    \frac{1}{n}
    \EE_{{(s',x')\sim \mu_{t,h}^{\pi}(\cdot\given s,x) 
    , a'\sim \pi_h(\cdot\given s',x')}}
    \sbr{
   \delta_{i, h}(s', x', a')
    } .
\end{align*}
Clearly, we have $({\hat V_{t,n}^{\pi,\theta}-\tilde V_t^{\pi,\theta}})(s,x)
=\sum_{h=t}^{T-1}\sum_{i=1}^n\Delta_{i,h}$.
Let 
$
\cG_{i,h}
=
\sigma (
\{Y_{\{j,k,s,a\}}: k\geq h+1,\ j\in[n],\ (s,a)\in\cS_k\times\cA_k\}
\cup
\{Y_{\{j,h,s,a\}}: j\ge i,\ (s,a)\in\cS_h\times\cA_h\}
)
$
denote the $\sigma$-algebra generated by the samples from stage $h$ after index $i$, and all samples starting from $h+1$.
Clearly, we have
\begin{align*}
    \EE\sbr{\Delta_{i,h}\given \cG_{i+1,h}}=
    0 .
\end{align*}
In addition, from Jensen's inequality, we have
\begin{align}
    \Var\rbr{\Delta_{i,h}\given \cG_{i+1,h}}
    & = \frac{1}{n^2} \EE \sbr{ \rbr{
    \EE_{{(s',x')\sim \mu_{t,h}^{\pi}(\cdot\given s,x) 
    , a'\sim \pi_h(\cdot\given s',x')}}
    \sbr{
   \delta_{i, h}(s', x', a')
    }
    }^2
    } \nonumber  \\
    & \leq 
    \frac{1}{n^2} \EE_{{(s',x')\sim \mu_{t,h}^{\pi}(\cdot\given s,x) 
    , a'\sim \pi_h(\cdot\given s',x')}} \sbr{ 
    \EE \sbr{
   \delta_{i, h}(s', x', a')
    }^2
    } \nonumber  \\
    & =
    \frac{1}{n^2}
    \EE_{(s',x')\sim \mu_{t,h}^{\pi}(\cdot\given s,x)}
    \sbr{
    \Var_{\tilde{\PP}_h^\pi(\cdot\given s',x')}
    \rbr{\hat V_{h+1,n}^{\pi,\theta}}
    } . \label{eval_variance_decomp_mdp}
\end{align}
Hence by combining \eqref{eval_variance_decomp_mdp} with the definition of $\hat{W}_{t,n}^{\pi, \theta}$ in \eqref{fin_mdp_pol_eval_err_var_sum_decom}, we obtain
\begin{align*}
    \tsum_{h=t}^{T-1}\tsum_{i=1}^n
    \Var\rbr{\Delta_{i,h}\given \cG_{i+1,h}}
    \le
    \frac{1}{n}\hat W_{t,n}^{\pi,\theta}(s,x).
\end{align*}
Combining the above observation with  $
\abr{\Delta_{i,h}}\leq \frac{2T L_C}{n}$ given that $|{\hat{V}^{\pi,\theta}_{h,n}}| 
\leq TL_C$, we can conclude the proof by invoking  Freedman's inequality. 
\end{proof}

Lemma \ref{lemma_fin_mdp_pol_eval_berns_sum} suggests that to control the error $(\hat{V}_{t,n}^{\pi, \theta} - \tilde{V}_t^{\pi, \theta})(s,x)$, it suffices to control the variance 
$\hat{W}_{t,n}^{\pi, \theta}$ defined as in \eqref{fin_mdp_pol_eval_err_var_sum_decom}.
It should be noted that the conditional variance  
$\Var_{\tilde{\PP}_h^\pi (\cdot|S_{h}, X_{h})} (\hat{V}_{h+1,n}^{\pi, \theta}) $ therein depends on the empirical value $\hat{V}_{h+1,n}^{\pi, \theta}$.
On the other hand, by utilizing the dynamic decomposition of total variance \cite{gheshlaghi2013minimax, azar2017minimax, agarwal2020model,sidford2018near}, we establish below that the counterpart of $\hat{W}_{t,n}^{\pi, \theta}$, defined by 
\begin{align*}
        {W}_{t}^{\pi, \theta}(s,x) =  \tsum_{h = t}^{T-1} \EE_{\tilde \PP}^\pi \left[ \left. \Var_{\tilde{\PP}_h^\pi (\cdot|S_{h}, X_{h})} (\tilde{V}_{h+1}^{\pi, \theta}) \right| (S_t, X_t) = (s,x) \right],
\end{align*}
scales at the order of $\cO(T^2)$.

\begin{lemma}
    \label{lemma_fin_mdp_pol_eval_bell_tot_var}
    For all $0 \leq t \leq T-1, s \in \cS_t$ and $ x \in \cX_t$, we have
    \begin{align*} 
    W_{t}^{\pi, \theta}(s,x)  \leq (T-t)^2 L_C^2. 
    \end{align*}
\end{lemma}
\begin{proof}
    Consider
    the variance of the cost-to-go defined by
    \begin{align}
    W_t^{\pi, \theta}(s,x) \coloneqq \Var_{\tilde \PP}^\pi \left( \left. \tsum_{h=t+1}^{T} \tilde{c}_h^{\pi, \theta}(S_h, X_h) \right| (S_t, X_t) = (s,x) \right), \nonumber
    \end{align}
    where $\tilde{c}_t^{\pi, \theta}(s,x) \coloneqq \EE_{a \sim \pi_t(\cdot|s,x)}[\tilde{c}_t^\theta(s,x,a)]$. Conditioning on $(S_{t+1}, X_{t+1})$ and using the law of total variance, we obtain
    \begin{align}
        W_t^{\pi, \theta}(s,x) = \Var_{\tilde{\PP}_t^\pi(\cdot|s,x)}(\tilde{V}_{t+1}^{\pi, \theta}) + \EE_{\tilde{\PP}_t^\pi(\cdot|s,x)} [W_{t+1}^{\pi, \theta}]. \label{fin_mdp_pol_eval_var_bell_rec}
    \end{align}
    Recursive application of \eqref{fin_mdp_pol_eval_var_bell_rec} then yields
    \begin{align}
    W_t^{\pi, \theta}(s,x) = \tsum_{h = t}^{T-1} \EE_{\tilde \PP}^\pi \left[ \left. \Var_{\tilde{\PP}_h^\pi (\cdot|S_{h}, X_{h})} (\tilde{V}_{h+1}^{\pi, \theta}) \right| (S_t, X_t) = (s,x) \right]. \nonumber
    \end{align}
    The proof is concluded by noting that $\abs{\tsum_{h=t+1}^{T} \tilde{c}_h^{\pi, \theta}(S_h, X_h)} \leq (T-t) L_C$ and hence $W_t^{\pi, \theta} \leq (T-t)^2 L_C^2$.
\end{proof}

With Lemma \ref{lemma_fin_mdp_pol_eval_val_lip}, \ref{lemma_fin_mdp_pol_eval_berns_sum}, and \ref{lemma_fin_mdp_pol_eval_bell_tot_var} in place, we are now ready to establish the sample complexity of estimating $\cR^\pi$ using $\hat{\cR}^\pi_n$.

\begin{theorem}\label{thrm_sample_fin_mdpite_eval}
Suppose $\pi$ satisfies \eqref{eq_finite_mdp_policy_lip}.
    For any $\epsilon > 0$ and $\delta \in (0,1)$, take 
    \begin{align*}
        n =  \mathcal{O}\left(\tfrac{T^2L_C^2}{\epsilon^2}\log \left(\tfrac{T^4 |\cS| L_C L_\pi (L_\Theta R_\Theta /\epsilon)^{d_\Theta}}{\delta \epsilon}\right)\right).
    \end{align*}
    Then with probability at least $1-\delta$,  we have
    \begin{align*} 
    \abs{\hat{\cR}_n^\pi(s) - \cR^\pi(s) } 
    \leq \epsilon,
    ~ \forall s \in \cS_0.
    \end{align*} 
\end{theorem}

\begin{proof}
Let us first  fix $\theta \in \Theta$.
    For $0 \leq t \leq T-1$, let $\cN_t$ denote an $\eta$-net for $[0, t]$ for some $\eta > 0$ to be determined later. Combining Lemma \ref{lemma_fin_mdp_pol_eval_berns_sum} with a union bound, we have with probability at least $1-\delta$ that
    \begin{equation}
        \label{fin_mdp_pol_eval_gen_bound_disct}
        \abs{ (\hat{V}_{t,n}^{\pi, \theta} - \tilde{V}_{t}^{\pi, \theta})(s,x_\eta)} \leq \sqrt{\tfrac{2}{n} \hat{W}_{t, n}^{\pi, \theta}(s,x_\eta)\log\left(\tfrac{2}{\delta'}\right)} + \tfrac{2 T L_C}{3n}  \log\left(\tfrac{2}{\delta'}\right),
    \end{equation}
    for all $0 \leq t \leq T, s \in \cS_t $ and $ x_\eta \in \cN_t$, where $\delta' = \tfrac{\delta}{T |\cS| |\cN_T|}$. Let us define
    \begin{align}
        \cB^{\pi,\theta}_n & \coloneqq \sup\left\{|(\hat{V}_{t,n}^{\pi, \theta} - \tilde{V}_{t}^{\pi, \theta})(s,x)| :0 \leq t \leq T, s \in \cS_t, x \in [0,t]\right\}, \label{fin_mdp_pol_eval_sup_def}\\
        \cB^{\pi,\theta}_n(\eta) & \coloneqq \sup\left\{|(\hat{V}_{t,n}^{\pi, \theta} - \tilde{V}_{t}^{\pi, \theta})(s,x_\eta) | : 0 \leq t \leq T, s \in \cS_t, x_\eta \in \cN_t\right\}. \label{fin_mdp_pol_eval_sup_disc_def}
    \end{align}
    From Lemma \ref{lemma_fin_mdp_pol_eval_val_lip}, we have
    \begin{align}
       \cB^{\pi,\theta}_n \leq &  2L_0^\pi \eta + \cB^{\pi,\theta}_n(\eta). \label{fin_mdp_pol_eval_gen_bound}
    \end{align}
    Combining the above observation with  \eqref{fin_mdp_pol_eval_gen_bound_disct}, it suffices to control $\hat{W}_{t, n}^{\pi, \theta}(s,x_\eta)$.
    To this end, we have
    \begin{align}
        \hat{W}_{t, n}^{\pi, \theta} (s,x) & = \tsum_{h = t}^{T-1} \EE_{\tilde \PP}^\pi \left[ \left. \Var_{\tilde{\PP}_h^\pi (\cdot|S_{h}, X_{h})} (\hat{V}_{h+1,n}^{\pi, \theta}) \right| (S_t, X_t) = (s,x) \right] \nonumber \\
        & \overset{(a)}{\leq} 2 \tsum_{h=t}^{T-1} \EE_{\tilde \PP}^\pi \left[ \left. \Var_{\tilde{\PP}_{h}^{\pi}(\cdot| S_{h}, X_{h})} ( \tilde{V}_{h+1}^{\pi, \theta}) \right| (S_t, X_t) = (s,x) \right] \nonumber \\
        & ~~~~ + 2 \tsum_{h = t}^{T-1} \EE_{\tilde \PP}^\pi \left[ \left. \EE_{\tilde{\PP}_{h}^{\pi}(\cdot| S_{h}, X_{h})} \left[(\hat{V}_{h+1,n}^{\pi, \theta} - \tilde{V}_{h+1}^{\pi, \theta})^2\right] \right| (S_t, X_t) = (s,x) \right] \label{elemental_variance_comparison} \\
        & \overset{(b)}{\leq} 2(T-t)^2 L_C^2 + 2(T-t)(\cB^{\pi,\theta}_n)^2, \label{fin_mdp_pol_eval_var_gen_bound}
    \end{align}
    where $(a)$ follows from the simple fact that  $\Var(Y + Z) \leq 2\Var(Y) + 2\EE[Z^2]$, and $(b)$ follows from Lemma \ref{lemma_fin_mdp_pol_eval_bell_tot_var}.
    Substituting \eqref{fin_mdp_pol_eval_var_gen_bound} into \eqref{fin_mdp_pol_eval_gen_bound_disct}, we obtain
    \begin{align*}
        \cB^{\pi,\theta}_n(\eta) & \leq \sqrt{\tfrac{2(2T^2 L_C^2 + 2T(\cB^{\pi,\theta}_n)^2)}{n}\log(\tfrac{2}{\delta'})} + \tfrac{2 T L_C}{3n} \log(\tfrac{2}{\delta'}) \\
        & \leq 2 T L_C \sqrt{\tfrac{1}{n} \log (\tfrac{2}{\delta'})} + 2 (2L_0^\pi \eta + \cB^{\pi,\theta}_n(\eta)) \sqrt{\tfrac{T}{n} \log (\tfrac{2}{\delta'})} + \tfrac{2T L_C}{3n}  \log(\tfrac{2}{\delta'}),
    \end{align*}
    where the last inequality uses \eqref{fin_mdp_pol_eval_gen_bound}. Suppose $n \geq 16T \log(\tfrac{2}{\delta'})$, then the above relation yields
    \begin{align}
    \cB^{\pi,\theta}_n \leq 4 T L_C \sqrt{\tfrac{1}{n} \log (\tfrac{2}{\delta'})} + 2L_0^\pi \eta + \tfrac{4T L_C}{3n}  \log(\tfrac{2}{\delta'}). \nonumber
    \end{align}
    Hence  it suffices to choose $\eta = \tfrac{\epsilon}{8 L_0^\pi}$, and
    \begin{align}
        n = \cO \left( \tfrac{T^2 L_C^2}{\epsilon^2} \log \left(\tfrac{1}{\delta'} \right)\right) \overset{(c)}{=} \cO \left( \tfrac{T^2 L_C^2}{\epsilon^2} \log \left(\tfrac{T^4 |\cS| L_C L_\pi}{\delta \epsilon} \right)\right) \label{fin_mdp_pol_eval_samp_comp_fix_theta}, 
    \end{align}
    from which we obtain 
    $
        \cB^{\pi,\theta}_n \leq \tfrac{3\epsilon}{4}.
    $
    Here $(c)$ follows from $|\cN_T| \leq \tfrac{ 8T L_0^\pi}{\epsilon}$, the definition of $\delta'$, and the fact that  Lemma \ref{lemma_fin_mdp_pol_eval_val_lip} implies $L_0^\pi = \cO(T^2 L_C L_\pi)$.
    It remains to control $\cB^{\pi,\theta}_n$ uniformly over all $\theta \in \Theta$.
    Let $\cN_\Theta$ denote the $\eta_\Theta$-net for $\eta_\Theta = \frac{\epsilon}{8L_\Theta}$. 
    Then applying the union bound, we have from \eqref{fin_mdp_pol_eval_samp_comp_fix_theta} that 
    $\sup_{\theta \in \cN_\Theta} \cB^{\pi,\theta}_n \leq \frac{3\epsilon}{4}$, provided  
    \begin{align*}
    n = \cO \left( \tfrac{T^2 L_C^2}{\epsilon^2} \log \left(\tfrac{T^4 |\cS| L_C L_\pi |\cN_\Theta|}{\delta \epsilon} \right)\right).
    \end{align*}
    Hence from Assumption \ref{assump_basic} and the choice of $\eta_\Theta$, we have
    $
        \cB^{\pi,\theta}_n \leq 2L_\Theta \eta_\Theta + B_{\pi, \theta_\eta} \leq \epsilon,
    $
    where $\theta_\eta \in \cN_\Theta$ is the closest point to $\theta$.
    The desired claim then follows from $|\cN_\Theta | \leq \left(1 +  {2 R_\Theta}/{\eta_\Theta}\right)^{d_\Theta}$.
\end{proof}

\begin{remark}\label{remark_lipschitz_policy}
It is worth discussing the role of condition \eqref{eq_finite_mdp_policy_lip} and its potential relaxation. 
From \eqref{fin_mdp_pol_eval_gen_bound_disct}, it is clear that Lemma \ref{lemma_fin_mdp_pol_eval_val_lip} is used to provide uniform control of $(\hat{V}_{t,n}^{\pi, \theta} - \tilde{V}_{t}^{\pi, \theta})(s,x)$ over $x \in \cX_t$ at every stage $t$, which in turn is a  consequence of condition \eqref{eq_finite_mdp_policy_lip}. 
It can be readily seen that for the purpose of uniformly controlling $(\hat{V}_{t,n}^{\pi, \theta} - \tilde{V}_{t}^{\pi, \theta})(s,x)$, one can instead control the error of estimating $\PP$ via $\hat{\PP}_n$ in $\ell_1$-distance. 
However, such an approach would clearly pay an additional factor of $\abs{\cS}$ in the final sample complexity.  
\end{remark}

In view of Theorem \ref{thrm_sample_fin_mdpite_eval}, the sample complexity for risk-averse policy evaluation with a static risk measure \eqref{eq_risk_functional} satisfying Assumption \ref{assump_basic} can be bounded on the order of 
$\tilde{\cO}(\frac{T^3  \abs{\cS} \abs{\cA} L_C^2 d_\Theta}{\epsilon^2} \log(\tfrac{L_\pi L_\Theta R_\Theta}{\delta \epsilon}))$.
Clearly, when $\cR^\pi (\cdot) = \EE^\pi \sbr{\cdot}$ corresponds to the expectation,  then one can take $L_C = 1$ and $f_\theta (z) = z$ in \eqref{eq_risk_functional}, and the obtained sample complexity recovers the optimal sample complexity for risk-neutral MDPs. 
In Section \ref{sec_application}, we will further instantiate Theorem \ref{thrm_sample_fin_mdpite_eval} with various concrete forms of risk functionals induced by $\phi$-divergences (Example \ref{example_phi_divergence_dro}), and establish some new sample complexity results for risk-averse MDPs.

\subsubsection{Risk-averse Policy Optimization}\label{sec_mdp_sample_finite_opt}
We now turn our attention to establishing the sample complexity for policy optimization \eqref{eq_risk_finite_horizon_mdp}.
Let us denote by $\hat{V}_{t,n}^\theta$ the optimal value function of the augmented MDP, defined in the same way as $\tilde{V}^\theta_t$ in  \eqref{dp_val_give_theta_finite}, with $\PP$ replaced by $\hat{\PP}_n$ therein.
In addition, we use  
\begin{align}\label{mdp_opt_risk_empirical}
    \cR^*(s) = \min_{\pi \in \Pi(\cM)} \cR^\pi(s), ~~~ \hat{\cR}^*_n(s) =  \min_{\pi \in \Pi(\cM)} \hat\cR_{n}^\pi(s)
\end{align}
to denote the optimal risk and its empirical estimate.
For the rest of our discussion in this section, we are interested in estimating the optimal risk $\cR^*$  via the estimator $\hat{\cR}^*_n$.
That is, for  any target precision $\epsilon > 0$, we are interested in the number of samples $n$ necessary to ensure 
$\sup_{s \in \cS_0} |{ \hat{\cR}^*_n(s) - \cR^*(s)}| \leq \epsilon$,
and that $\hat{\pi}^*_n \in \Argmin_{\pi} \hat{\cR}_n^\pi(s)$ is an $\epsilon$-optimal policy for \eqref{eq_risk_finite_horizon_mdp} with high probability.

Before we proceed, it could be worth discussing our basic idea moving forward. 
Clearly, we have 
\begin{align*}
    \cR^*(s) = \min_{\theta \in \Theta} \min_{\pi \in \Pi(\cM)} \tilde{V}^{\pi, \theta}_0(s,0) + f_\theta(0), ~~~
    \hat{\cR}_n^*(s) = \min_{\theta \in \Theta} \min_{\pi \in \Pi(\cM)} \hat{V}^{\pi, \theta}_{0,n}(s,0) + f_\theta(0).
\end{align*}
Hence it suffices to control 
$
\cE_n = {\min_{\pi \in \Pi(\cM)} \tilde{V}^{\pi, \theta}_0(s,0) - \min_{\pi \in \Pi(\cM)} \hat{V}^{\pi, \theta}_{0,n}(s,0)}
$ uniformly for all $\theta$.
To this end, let $\pi^\theta \in \Argmin_{\pi} \tilde{V}^{\pi, \theta}_0(s,0)$
and $\hat{\pi}^\theta_n \in \Argmin_{\pi} \hat{V}^{\pi, \theta}_{0,n}(s,0) $.
We  have 
\begin{align}\label{ineq_policy_opt_mdp_basic}
     \tilde{V}^{\pi^\theta, \theta}_0(s,0) - \hat{V}^{\pi^\theta, \theta}_{0,n}(s,0) \overset{(a)}{\leq}  
    \cE_n \overset{(b)}{\leq}  \tilde{V}^{\hat{\pi}^\theta_n, \theta}_0(s,0) - \hat{V}^{\hat{\pi}^\theta_n, \theta}_{0,n}(s,0).
\end{align}
It is then immediate that the left-hand side of $(a)$ can be readily controlled in view of Theorem \ref{thrm_sample_fin_mdpite_eval}, provided that the optimal policy $\pi^\theta$ is Lipschitz continuous w.r.t the augmented state variable (i.e., \eqref{eq_finite_mdp_policy_lip}). 
Unfortunately, this is clearly not the case given that $\pi^\theta$ is non-randomized (Proposition  \ref{thrm_dp_fix_theta_finite}). 
Nevertheless, we proceed to show that one can construct a close-to-optimal policy without paying too much of a price on its Lipschitz constant. 
To this end, we first establish the Lipschitz continuity of the optimal value function  $\tilde{V}^\theta_t$
with respect to the augmented state variable.

\begin{lemma}{\label{lemma_fin_mdp_pol_opt_val_lip}
   For any $0 \leq t \leq T$ and $s \in \cS_t$, the value functions $\tilde{V}_t^\theta(s, \cdot)$ and $\hat{V}^\theta_{t,n}(s, \cdot)$ are Lipschitz with modulus
    \begin{align*}
         L_t = 2(T-t+1)L_C.
    \end{align*}
}    
\end{lemma}

\begin{proof}
As in Lemma \ref{lemma_fin_mdp_pol_eval_val_lip}, we proceed by backward induction. For $t = T$ and any $s \in \cS_T$, we have
\begin{align*}
    | \tilde V_T^\theta(s,x) - \tilde V_T^\theta(s,y)| & = |\min_{a \in \cA_T} \tilde{c}_T^\theta(s,x,a) - \min_{a \in \cA_T} \tilde{c}_T^\theta(s,y,a)| 
     \leq 2L_C|x-y|,
\end{align*}
where the last inequality follows from Assumption \ref{assump_basic}. For $0 \leq t \leq T-1$, we obtain
\begin{align*}
     & |\tilde{V}_t^\theta(s,x) - \tilde{V}_t^\theta(s,y)| \\
     \leq & \max_{a \in \cA_t}|\tilde{c}_t^\theta(s,x,a) -  \tilde{c}_t^\theta(s,y,a) + \mathbb{E}_{s' \sim \PP_t(\cdot|s,a)}[\tilde{V}_{t+1}^\theta(s', x + c_t(s,a)) - \tilde{V}_{t+1}^\theta(s', y + c_t(s,a))]| \\
     \leq & 2L_C|x-y| + 2(T-t)L_C|x-y|,
\end{align*}
where the last inequality applies the inductive hypothesis.
\end{proof}
With Lemma \ref{lemma_fin_mdp_pol_opt_val_lip}, we can establish that for any $\epsilon > 0$, there exists an $\epsilon$-optimal  randomized policy $\pi$  that satisfies  \eqref{eq_finite_mdp_policy_lip} with a Lipschitz constant $L_{\pi} = \cO(\log(1/\epsilon)/\epsilon)$.

\begin{lemma}\label{lemma_close_to_opt_lip_policy_mdp}
    For any $\epsilon > 0$, there exists a randomized policy $\pi$ such that 
    $
            \tilde{V}^{\pi, \theta}_t(s,x) -  \tilde{V}^{ \theta}_t(s,x)  \leq \epsilon$
    and $
        \norm{
        \pi_t(\cdot|s,x) -  \pi_t(\cdot|s,x') 
        }_1 
         \leq 
        L_\pi(\epsilon) \abs{x - x'}
    $, 
where 
\begin{align}\label{lipschitz_const_approx_opt}
L_\pi(\epsilon) = \frac{4 (T+1)  (T+2) L_C}{\epsilon}
        \log \rbr{
        \frac{2  \abs{\cA} (T+1)^2 L_C }{\epsilon }
        } .
\end{align}
\end{lemma}

\begin{proof}
We first show that the optimal Q-function, defined as 
\begin{align*}
    \tilde{Q}_t^\theta(s,x,a) \coloneqq \tilde{c}_t^\theta(s, x, a) + \tsum_{s' \in \cS_{t+1}} \PP_t(s' | s, a) \tilde{V}^\theta_{t+1}(s', x + c_t(s, a)),
\end{align*}
is Lipschitz continuous in $x$ for every $(s, a) \in \cS_t \times \cA_t$.
To see this, note that 
\begin{align}\label{ineq_lipschitz_optimal_q}
    & \abs{
    \tilde{Q}_t^\theta(s,x,a) - \tilde{Q}_t^\theta(s,x',a)
    } \nonumber  \\
      \leq &  \abs{
    \tilde{c}_t^\theta(s, x, a) - \tilde{c}_t^\theta(s, x', a)
    }   + \sup_{s'} \abs{
    \tilde{V}^\theta_{t+1}(s', x + c_t(s, a))
    - \tilde{V}^\theta_{t+1}(s', x' + c_t(s, a))
    } \nonumber \\
    \overset{(a)}{\leq}  & 
    2 (T-t+2) L_C \abs{x - x'}
\end{align}
where $(a)$ follows from Lemma \ref{lemma_fin_mdp_pol_opt_val_lip}, the definition of $\tilde{c}_t^\theta$, and Assumption \ref{assump_basic}.
Now for any $\eta >0$, consider the policy $\pi^\eta$ defined as 
$
    \pi_t^\eta(s,x, \cdot) = \mathrm{Softmax}_\eta ({ - \tilde{Q}^\theta_t(s,x, \cdot) }),
$
where $
[\mathrm{Softmax}_\eta (q)]_i = \exp( \eta q_i) /\tsum_{j} \exp( \eta q_j)
$.
Clearly, we have 
$
\norm{\mathrm{Softmax}_\eta (q) - \mathrm{Softmax}_\eta (q')}_1 \leq \eta \norm{q - q'}_\infty,
$
and hence
\begin{align}\label{ineq_approx_opt_policy_lip}
    \norm{\pi_t^\eta(s, x, \cdot) - \pi_t^\eta(s, x', \cdot ) }_1 
    \leq \eta \norm{\tilde{Q}^\theta_t(s,x, \cdot) - \tilde{Q}^\theta_t(s,x', \cdot)}_\infty
    \overset{(b)}{\leq} 2 \eta (T-t+2) L_C \abs{x - x'},
\end{align}
where $(b)$ follows from \eqref{ineq_lipschitz_optimal_q}.
On the other hand, it holds that
\begin{align}
    \tilde{V}^{\pi^\eta, \theta}_0(s,x) -  \tilde{V}^{ \theta}_0(s,x)
    & \overset{(c)}{=} \EE^{\pi^\eta}_{\tilde{\PP}}
    \sbr{
        \tsum_{t=0}^T  \tsum_{a \in \cA_t} \pi_t(a|S_t, X_t)\rbr{  \tilde{Q}_t^\theta (S_t, X_t, a) - \min_{a\in \cA_t} \tilde{Q}_t^\theta(S_t, X_t, a) 
        }
    }  \nonumber \\ 
    & \leq 
     \EE^{\pi^\eta}_{\tilde{\PP}}
    \sbr{
        \tsum_{t=0}^T \rbr{\tsum_{a \in \cA_t^\epsilon(S_t, X_t) } \pi_t(a|S_t, X_t) \epsilon 
        + \tsum_{a \notin \cA_t^\epsilon(S_t, X_t) }  (T+1) L_C \exp(-\eta \epsilon)
        }
    } 
    \nonumber \\
    & \overset{(d)}{\leq}   (T+1) \sbr{ \epsilon 
    + \abs{\cA} (T+1) L_C \exp(-\eta \epsilon) }
    \overset{(e)}{\leq} 2 (T+1) \epsilon,
    \label{ineq_approx_opt_policy_gap}
\end{align}
where 
$\cA_t^\epsilon (s,x)= \{a \in \cA_t: 
\tilde{Q}^\theta_t(s,x,a) \leq 
\min_{a \in \cA_t} \tilde{Q}^\theta_t(s,x,a) + \epsilon 
\}$, 
 $(c)$ follows from the performance difference lemma \cite{kakade2002approximately},
 $(d)$ follows from the definition of $\mathrm{Softmax}_\eta$ and $\cA_t^\epsilon(s,x)$, 
 and $(e)$ follows from the choice of 
$
\eta =  \log( \max_t \abs{\cA_t} (T+1) L_C /\epsilon ) / \epsilon .
$
The desired claim then follows by replacing $\epsilon$
with $\epsilon / [2(T+1)]$ in \eqref{ineq_approx_opt_policy_gap},
while utilizing \eqref{ineq_approx_opt_policy_lip}.
\end{proof}

With Lemma \ref{lemma_close_to_opt_lip_policy_mdp} in place, one can replace $\pi^\theta$ in \eqref{ineq_policy_opt_mdp_basic} with a Lipschitz continuous policy $\pi$, and subsequently control the left-hand side of $(a)$ therein via Theorem \ref{thrm_sample_fin_mdpite_eval}.
Our remaining discussion will focus on controlling the right-hand side of $(b)$ in \eqref{ineq_policy_opt_mdp_basic} instead. 
In particular, we seek to control $|   \tilde{V}^{\pi, \theta}_0(s) - \hat{V}^{\pi, \theta}_{0,n}(s) |$ for potentially data-dependent policy $\pi$.

We proceed to establish a counterpart of Lemma \ref{lemma_err_bernstein_eval} in the context of  policy optimization. 
Specifically, Lemma \ref{lemma_fin_mdp_pol_opt_bern_rad} controls the error $ |{ (\hat{V}_{t,n}^{\pi, \theta} - \tilde{V}_t^{\pi, \theta})(s,x)} |$ via the  variance
$\Var_{\tilde{\PP}_t(\cdot|s,x,a)} (\hat{V}_{t+1,n}^{\pi, \theta})$, for some policy $\pi$ that potentially depends on the samples.
It should be noted that in contrast to policy evaluation (Lemma \ref{lemma_err_bernstein_eval}),  due to the policy $\pi$ being sample-dependent, the martingale structure in the error decomposition \eqref{fin_mdp_pol_eval_err_sum_decom} is lost for policy optimization. 
Consequently, we proceed with the following approach. 

\begin{lemma}\label{lemma_fin_mdp_pol_opt_bern_rad}
    Let $\pi$ be a potentially sample-dependent policy  satisfying \eqref{eq_finite_mdp_policy_lip}. For any $\eta > 0$ and $\delta \in (0,1)$, with probability at least $1-\delta$, we have 
    \begin{align}
        \abs{ (\hat{V}_{t,n}^{\pi, \theta} - \tilde{V}_t^{\pi, \theta})(s,x)} \leq & \EE_{\tilde{\PP}}^\pi \left[\left. \tsum_{h=t}^{T-1}b_{h,n}^{\pi, \theta}(S_h, X_h, A_h)\right| (S_t, X_t) = (s,x) \right] \nonumber \\
        & ~~~  + 4TL_{0}^\pi \eta + \sqrt{\tfrac{8T^3L_C L_{0}^\pi \eta}{n} \log(\tfrac{2T^2 |\cS| |\cA|}{\delta \eta})},  
            \end{align}
   for any $0 \leq t \leq T-1$, $s \in \cS_t$, and $ x \in [0,t]$, 
    where
 \begin{align}\label{eq_def_local_bernstein}
  b_{t,n}^{\pi, \theta}(s,x,a) \coloneqq \sqrt{\tfrac{2}{n}\Var_{\tilde{\PP}_t(\cdot|s,x,a)} (\hat{V}_{t+1,n}^{\pi, \theta}) \log(\tfrac{2T^2 |\cS| |\cA|}{\delta \eta})} + \tfrac{2T L_C}{3n} \log(\tfrac{2T^2 |\cS| |\cA|}{\delta \eta}).
  \end{align}
\end{lemma}
\begin{proof}
    Let $\cN_t$ be an $\eta$-net for the interval $[0,t]$. Applying Bernstein's inequality and taking a union bound yields
    \begin{equation}
    \label{fin_mdp_pol_opt_gen_bound_disc}
        \abs{(\EE_{\hat{\PP}_{t,n} (\cdot|s,x_\eta,a)} - \EE_{\tilde{\PP}_t (\cdot|s,x_\eta,a)})[\hat{V}_{t+1,n}^{\pi, \theta}]} \leq b_{t,n}^{\pi, \theta}(s,x_\eta, a), 
    \end{equation}
for any $ 0 \leq t \leq T-1$,  $s \in \cS_t$, $x_\eta \in \cN_t$,  $a \in \cA_t$ with probability at least $1- \delta$.
    For each $x \in [0,t]$, let $x_\eta$ be its closest point on $\cN_t$. We have
    \begin{align}
        & \abs{(\hat{V}_{t,n}^{\pi, \theta} - \tilde{V}_{t}^{\pi, \theta})(s,x)}\nonumber \\
        \overset{(a)}{\leq} & 2L_{0}^\pi \eta + \abs{\EE_{a \sim \pi_t (\cdot|s,x)} \left[ (\EE_{\hat{\PP}_{t,n}(\cdot|s,x_\eta,a)} - \EE_{\tilde{\PP}_{t}(\cdot|s,x_\eta,a)}) [\hat{V}_{t+1,n}^{\pi, \theta}] + \EE_{\tilde{\PP}_t (\cdot|s,x_\eta,a)}[\hat{V}_{t+1,n}^{\pi, \theta} - \tilde{V}_{t+1}^{\pi, \theta}] \right]} \nonumber \\
        \overset{(b)}{\leq} & 2 L_{0}^\pi \eta + \EE_{a \sim \pi_t (\cdot|s,x)} \left[ b_{t,n}^{\pi, \theta}(s,x_\eta, a) + \EE_{\tilde{\PP}_t (\cdot|s,x_\eta,a)}\left[\abs{\hat{V}_{t+1,n}^{\pi, \theta} - \tilde{V}_{t+1}^{\pi, \theta}}\right] \right]  \label{ineq_opt_val_err_uniform_raw}
            \end{align}
    where $(a)$ follows from Lemma \ref{lemma_fin_mdp_pol_eval_val_lip}, and $(b)$ follows from \eqref{fin_mdp_pol_opt_gen_bound_disc}.
    
To proceed, let us first show that the conditional variance $\Var_{\tilde{\PP}_t(\cdot|s,x,a)} (\hat{V}_{t+1, n}^{\pi, \theta})$ is indeed Lipschitz in the augmented state variable $x$.
    Since the augmented transition kernel is translation-equivariant in $x$, we have
    \begin{align}
    & \abs{ \Var_{s',x' \sim \tilde{\PP}_t (\cdot|s,x,a)}(\hat{V}_{t+1,n}^{\pi, \theta}(s',x')) -   \Var_{s',x' \sim \tilde{\PP}_t (\cdot|s,y,a)}(\hat{V}_{t+1,n}^{\pi, \theta}(s',x'))  }  \nonumber \\
   =     & \abs{\Var_{s',x' \sim \tilde{\PP}_t (\cdot|s,x,a)}(\hat{V}_{t+1,n}^{\pi, \theta}(s',x')) - \Var_{s',x' \sim \tilde{\PP}_t (\cdot|s,x,a)}(\hat{V}_{t+1,n}^{\pi, \theta}(s',x' - (x-y)))} \nonumber \\
        \leq &  4 (T-t) L_C L_{t+1}^\pi \abs{x-y}, \label{ineq_var_lipschitz_eval}
    \end{align}
    where the last inequality follows from a direct application of Lemma \ref{lemma_fin_mdp_pol_eval_val_lip}.
Combining \eqref{ineq_opt_val_err_uniform_raw}, \eqref{ineq_var_lipschitz_eval}, and the definition of $b_{t,n}^{\pi, \theta}$ in \eqref{eq_def_local_bernstein}, we can conclude that 
  \begin{align*}
         & \abs{(\hat{V}_{t,n}^{\pi, \theta} - \tilde{V}_{t}^{\pi, \theta})(s,x)}\nonumber \\
       \leq & 4L_{0}^\pi \eta + \sqrt{\tfrac{8 T L_C L_{0}^\pi \eta}{n} \log(\tfrac{2T^2 |\cS| |\cA|}{\delta \eta})} +  \EE_{a \sim \pi_t (\cdot|s,x)} \left[ b_{t,n}^{\pi, \theta}(s,x, a) + \EE_{\tilde{\PP}_t (\cdot|s,x,a)}\left[\abs{\hat{V}_{t+1,n}^{\pi, \theta} - \tilde{V}_{t+1}^{\pi, \theta}} \right] \right].
\end{align*}
 The desired claim then follows from the recursive application of the above observation.
\end{proof}

With Lemma \ref{lemma_close_to_opt_lip_policy_mdp}, \ref{lemma_fin_mdp_pol_opt_val_lip}, and \ref{lemma_fin_mdp_pol_opt_bern_rad} in place, we are ready to establish the sample complexity for estimating the optimal risk $\cR^*$ of the risk-averse policy optimization problem \eqref{eq_risk_finite_horizon_mdp}.
We will also discuss how the result readily implies a procedure to produce an $\epsilon$-optimal policy for \eqref{eq_risk_finite_horizon_mdp}, which simultaneously satisfies the Lipschitz  property of \eqref{eq_finite_mdp_policy_lip}.

\begin{theorem}\label{thrm_sample_fin_mdpite_opt}
For any $\epsilon > 0$ and $\delta \in (0,1)$, take
\begin{align}
    n = \cO\left( \tfrac{T^3 L_C^2}{\epsilon^2} \log \left(\tfrac{T^7 |\cS| |\cA| L_C^3 \log(T^2|\cA|L_C/\epsilon) (L_\Theta R_\Theta/\epsilon)^{d_\Theta}}{\delta \epsilon^3} \right)\right). \nonumber
\end{align}
%
Then with probability at least $1-\delta$, we have
\begin{align}
    \abs{ \hat{\cR}^*_n (s) - \cR^* (s)  } \leq \epsilon, ~  \forall s \in \cS_0. \nonumber
\end{align}
\end{theorem}

\begin{proof}
Let us first fix $\theta \in \Theta$.  From Lemma \ref{lemma_close_to_opt_lip_policy_mdp}, there exists an $L_\pi({\epsilon})$-Lipschitz policy $\pi_\epsilon^\theta$ (resp. $\hat{\pi}_\epsilon^\theta$),  such that
\begin{align}\label{ineq_choice_epsilon_opt_policy}
    \tilde{V}^{\pi_\epsilon^\theta, \theta}_t(s,x) -  \tilde{V}^{ \theta}_t(s,x)  \leq \frac{\epsilon}{4}, ~~~
    \hat{V}^{\hat{\pi}_\epsilon^\theta, \theta}_{t,n}(s,x) -  \hat{V}^{ \theta}_{t,n}(s,x)  \leq \frac{\epsilon}{4},
\end{align}
where $L_\pi(\epsilon)$ is defined as in \eqref{lipschitz_const_approx_opt}.
Consequently, we have
\begin{align}
    \hat{V}_{t,n}^{\theta}(s,x) - \tilde{V}_t^{\theta}(s,x) & \geq (\hat{V}_{t,n}^{\hat{\pi}_\epsilon^\theta, \theta}(s,x) - \frac{\epsilon}{4}) - \tilde{V}_t^{\hat{\pi}_\epsilon^\theta, \theta}(s,x),  \label{ineq_opt_err_at_emp_policy} \\
    \hat{V}_{t,n}^{\theta}(s,x) - \tilde{V}_t^{\theta}(s,x) & \leq \hat{V}_{t,n}^{\pi_\epsilon^\theta, \theta}(s,x) - (\tilde{V}_t^{\pi_\epsilon^\theta, \theta}(s,x) - \frac{\epsilon}{4}). \label{ineq_opt_err_at_true_policy}
\end{align}
Since $\pi_\epsilon^\theta$ does not depend on the samples, it follows   from Lemma \ref{lemma_close_to_opt_lip_policy_mdp},  \eqref{ineq_opt_err_at_true_policy} and  Theorem \ref{thrm_sample_fin_mdpite_eval} that 
\begin{align}\label{ineq_opt_sample_pi_true}
 \hat{V}_{t,n}^{\theta}(s,x) - \tilde{V}_t^{\theta}(s,x) \leq    \hat{V}_{t,n}^{\pi_\epsilon^\theta, \theta}(s,x) - \tilde{V}_t^{\pi_\epsilon^\theta, \theta}(s,x)  + \frac{\epsilon}{4}
    \leq  \frac{\epsilon}{2},
\end{align}
provided 
$n 
= \cO \rbr{
\tfrac{T^2L_C^2}{\epsilon^2}\log \left(\tfrac{T^4 |\cS| L_C L_\pi(\epsilon) (L_\Theta R_\Theta /\epsilon)^{d_\Theta}}{\delta \epsilon}\right)
}
$.

It remains to control 
$\hat{V}_{t,n}^{\hat{\pi}_\epsilon^\theta, \theta}(s,x) - \tilde{V}_t^{\hat{\pi}_\epsilon^\theta,\theta}(s,x)$
and subsequently \eqref{ineq_opt_err_at_emp_policy}. For a given $\eta > 0$ to be specified later, let $\cN_t$ be an $\eta$-net of $[0,t]$ for $0 \leq t \leq T$. Going forward, let us denote $\pi = \hat{\pi}_\epsilon^\theta$. From Lemma \ref{lemma_fin_mdp_pol_opt_bern_rad}, we obtain
\begin{align}
    \abs{ (\hat{V}_{t,n}^{\pi, \theta} - \tilde{V}_t^{\pi, \theta})(s,x)} \leq & \EE_{\tilde \PP}^\pi \left[\left. \tsum_{h=t}^{T-1}b_{h,n}^{\pi, \theta}(S_h, X_h, A_h)\right| (S_t, X_t) = (s,x) \right] \nonumber  \\
    & + 4TL_{0}^\pi \eta + \sqrt{\tfrac{8T^3L_C L_{0}^\pi(\epsilon) \eta}{n} \log(\tfrac{2}{\delta'}}),
    \label{fin_mdp_pol_opt_gen_bound}
    \end{align}
where $\delta' = \tfrac{\delta \eta}{T^2 |\cS| |\cA|}$,
and $L^\pi_t(\epsilon)$ is defined as in Lemma \ref{lemma_fin_mdp_pol_eval_val_lip}, with $L_\pi$ replaced by $L_\pi({\epsilon})$ therein.
We proceed to control $b^{\pi, \theta}_{h,n}$ above. 
From the Cauchy-Schwarz inequality, it follows that
\begin{align}
     & \EE_{\tilde \PP}^\pi\left[ \left. \tsum_{h=t}^{T-1} \sqrt{\tfrac{2}{n} \Var_{\tilde{\PP}_{h}(\cdot|S_h, X_h, A_h)}(\hat{V}_{h+1,n}^{\pi, \theta}) \log(\tfrac{2}{\delta'})} \right| (S_t, X_t) = (s,x)\right] \nonumber \\
    \leq & \sqrt{\tfrac{2T}{n} \log(\tfrac{2}{\delta'}) \EE_{\tilde \PP}^\pi \left[ \left. \tsum_{h=t}^{T-1} \Var_{\tilde{\PP}_{h}(\cdot|S_h, X_h, A_h)}(\hat{V}_{h+1,n}^{\pi, \theta}) \right| (S_t, X_t) = (s,x)\right] }. \label{fin_mdp_pol_opt_bern_rad_fir_term}
\end{align}
Following the same argument as in  \eqref{fin_mdp_pol_eval_var_gen_bound}, we have
\begin{align}
    \EE_{\tilde \PP}^\pi \left[ \left. \tsum_{h=t}^{T-1} \Var_{\tilde{\PP}_{h}(\cdot|S_h, X_h, A_h)}(\hat{V}_{h+1,n}^{\pi, \theta}) \right| (S_t, X_t) = (s,x)\right] \leq 2(T-t)^2 L_C^2 + 2(T-t) (\cB^{\pi,\theta}_n)^2, \label{fin_mdp_pol_opt_bern_rad_fir_term_2}
\end{align}
where  $\cB^{\pi,\theta}_n$ is defined as in \eqref{fin_mdp_pol_eval_sup_def}. 
Combining \eqref{fin_mdp_pol_opt_bern_rad_fir_term}, \eqref{fin_mdp_pol_opt_bern_rad_fir_term_2}, and the definition of $b^{\pi,\theta}_{h,n}$ in \eqref{eq_def_local_bernstein}, we have 
\begin{align*}
    \EE_{\tilde \PP}^\pi \left[\left. \tsum_{h=t}^{T-1}b_{h,n}^{\pi, \theta}(S_h, X_h, A_h)\right| (S_t, X_t) = (s,x) \right] 
    \leq 
    \sqrt{\tfrac{2  T (2T^2 L_C^2 + 2T (\cB^{\pi,\theta}_n)^2) }{n} \log(\tfrac{2}{\delta'})}
    + \tfrac{2T^2 L_C }{3n}\log(\tfrac{2}{\delta'}).
\end{align*}
The above observation, together with \eqref{fin_mdp_pol_opt_gen_bound},
implies 
\begin{align}
    \cB^{\pi,\theta}_n & \leq \sqrt{\tfrac{2  T (2T^2 L_C^2 + 2T (\cB^{\pi,\theta}_n)^2) }{n} \log(\tfrac{2}{\delta'})} + \tfrac{2T^2 L_C }{3n}\log(\tfrac{2}{\delta'}) + 4 T L_0^\pi(\epsilon) \eta + \sqrt{\tfrac{8T^3 L_C L_0^\pi(\epsilon) \eta }{n} \log(\tfrac{2}{\delta'})}. \label{fin_mdp_pol_opt_B_final_ineq}
\end{align}
Suppose $n \geq 16 T^2 \log(\tfrac{2}{\delta'})$, then a simple rearrangement of the above inequality yields
\begin{align*}
    \cB^{\pi,\theta}_n & \leq 4 T^{3/2} L_C\sqrt{\tfrac{1}{n} \log (\tfrac{2}{\delta'})} + \tfrac{4T^2 L_C }{3n} \log(\tfrac{2}{\delta'}) + 8 T L_0^\pi \eta + \sqrt{2 T L_C L_0^\pi (\epsilon) \eta}.
\end{align*}
Thus, in order to ensure $\cB^{\pi,\theta}_n \leq \tfrac{\epsilon}{2}$, it suffices to choose $\eta = \tfrac{\epsilon^2}{32 T L_C L_0^\pi (\epsilon)}$ and
\begin{align}
    n  = \cO\left(\tfrac{T^3 L_C^2}{\epsilon^2} \log \left(\tfrac{1}{\delta'} \right)\right) \overset{(a)}{=} \cO \left( \tfrac{T^3 L_C^2}{\epsilon^2} \log \left(\tfrac{T^5 |\cS| |\cA| L_C^2 L_\pi (\epsilon)}{\delta \epsilon^2} \right) \right)
\overset{(b)}{=} \cO \left( \tfrac{T^3 L_C^2}{\epsilon^2} \log \left(\tfrac{T^7 |\cS| |\cA| L_C^3 \log(T^2|\cA|L_C/\epsilon)}{\delta \epsilon^3} \right) \right)
    \nonumber ,
\end{align}
where $(a)$ uses the definition of $\delta'$ and $\eta$, together with Lemma \ref{lemma_fin_mdp_pol_eval_val_lip},
which implies $L_0^\pi(\epsilon) = \cO(T^2L_CL_\pi(\epsilon))$,
and $(b)$ follows from the definition of $L_\pi(\epsilon)$ in \eqref{lipschitz_const_approx_opt}.
To conclude our proof, it suffices to control $\cB^{\pi,\theta}_n$ uniformly over $\theta \in \Theta$,
 and this follows from  the same covering argument as in Theorem \ref{thrm_sample_fin_mdpite_eval}.
\end{proof}

In view of Theorem \ref{thrm_sample_fin_mdpite_opt}, the sample complexity for estimating the optimal risk $\cR^*$ up to an $\epsilon$-accuracy can be bounded by $\tilde{\cO}({\frac{T^4 \abs{\cS} \abs{\cA} L_C^2 d_\Theta}{\epsilon^2} \log(\tfrac{L_\Theta R_\Theta}{\delta \epsilon}) } )$.
Clearly, when $\cR^\pi(\cdot) = \EE^\pi \sbr{\cdot}$, one can take $L_C = 1$ and $f_\theta(z) = z$ in \eqref{value_given_theta_finite_horizon_mdp}, and the obtained sample complexity recovers the optimal sample complexity for policy optimization in risk-neutral MDPs. 
We will further discuss the application of 
Theorem \ref{thrm_sample_fin_mdpite_opt} when $\cR^\pi$ corresponds to the distributionally robust functional induced by $\phi$-divergences.

\begin{remark}
    It is also worth noting that in addition to estimating the optimal risk $\cR^*$,  Theorem \ref{thrm_sample_fin_mdpite_opt} also produces an $(\frac{3\epsilon}{2})$-optimal policy for the policy optimization problem \eqref{eq_risk_finite_horizon_mdp}.
In particular, we have 
\begin{align*}
   \tilde{V}_t^{\hat{\pi}_\epsilon^\theta, \theta}(s,x)      \overset{(a)}{\leq} \hat{V}_{t,n}^{\hat{\pi}_\epsilon^\theta, \theta}(s,x) +  \cB^{\hat{\pi}_\epsilon^\theta,\theta}_n   \overset{(b)}{\leq} \hat{V}_{t,n}^{\theta}(s,x) + \frac{\epsilon}{4} + \cB^{\hat{\pi}_\epsilon^\theta,\theta}_n  \overset{(c)}{\leq} \hat{V}_{t,n}^{\pi_\epsilon^\theta, \theta}(s,x) + \frac{\epsilon}{4} + \cB^{\hat{\pi}_\epsilon^\theta,\theta}_n 
\end{align*}
where $(a)$ follows from the definition of $\cB^{\pi \theta}_n$, 
$(b)$ follows from \eqref{ineq_choice_epsilon_opt_policy},
and $(c)$ follows from the definition of $\hat{V}^\theta_{t,n}$. 
Applying \eqref{ineq_choice_epsilon_opt_policy} and  the definition of $\cB^{\pi \theta}_n$ again, we obtain 
\begin{align*}
      \tilde{V}_t^{\hat{\pi}_\epsilon^\theta, \theta}(s,x) 
       \leq \tilde{V}_{t}^{\pi_\epsilon^\theta, \theta}(s,x) + \frac{\epsilon}{4} + \cB^{\hat{\pi}_\epsilon^\theta,\theta}_n + \cB^{\pi_\epsilon^\theta,\theta}_n
      \leq \tilde{V}^\theta_t + \frac{\epsilon}{2} + \cB^{\hat{\pi}_\epsilon^\theta,\theta}_n + \cB^{\pi_\epsilon^\theta,\theta}_n
  \leq \tilde{V}^\theta_t + \frac{3 \epsilon}{2},
\end{align*}
where the last inequality uses $\cB^{\hat{\pi}_\epsilon^\theta ,\theta}_n + \cB^{\pi_\epsilon^\theta ,\theta}_n\leq \epsilon$.
Since the above holds
uniformly for all $\theta \in \Theta$,
from Theorem \ref{cor_dp_optimal_value_finite_mdp}, 
we obtain   
$
\cR^{\hat{\pi}_\epsilon^{\hat{\theta}^*}} (s) \leq \cR^* (s) + \frac{3\epsilon}{2},
$
where 
$\hat{\theta}^* \in \Argmin_{\theta \in \Theta} \tilde{V}_0^{\hat{\pi}_\epsilon^\theta, \theta}(s,0) + f_\theta(0)$.
\end{remark}

%% file: risk_finite_soc.tex

\section{Risk-averse Stochastic Optimal Control}\label{sec_soc}

In this section, we turn our attention to risk-averse SOC problems and show that the approach presented in Section \ref{sec_risk_mdp} can be naturally extended to the SOC model. 
Similar to MDPs, for any SOC instance $\cM$, we use $\Pi(\cM)$ to denote the set of randomized history-dependent policies,
and $\Pi_{\mathrm{MR}}(\cM)$ for the set of randomized Markovian policies.

\subsection{Finite-horizon Risk-averse Stochastic Optimal Control}

Consider a finite-horizon SOC problem 
$\cM = (
\cbr{\cS_t}_{t=0}^T, \cbr{\cA_t}_{t=0}^T, \cbr{P_t}_{t=0}^{T-1},
\cbr{F_t}_{t=0}^{T-1}, \cbr{c_t}_{t=0}^T
)$.
Here $\cS_t \subseteq \RR^{d_\cS}$ denotes the compact state space at stage $t$, 
$\cA_t \subseteq \RR^{d_\cA}$ denotes its corresponding compact control space, 
$c_t: \cS_t \times \cA_t \to [0,1]$ denotes the continuous cost function, and 
$F_t: \cS_t \times \cA_t \times \Xi_t \to \cS_{t+1}$ denotes the continuous transition function.
The transition noise $\xi_t \in \Xi_t$ follows  the distribution $P_t$, and $\cbr{\xi_t}_{t=0}^{T-1}$ are independent.
We further assume that $P_t$ does not depend on the state or action.
For notational convenience, we will occasionally write $P$ as shorthand for $\cbr{P_t}_{t=0}^{T-1}$.
We seek to find the optimal policy of the following risk,
\begin{align}\label{eq_risk_finite_horizon_soc}
    \min_{\pi \in \Pi(\cM)} \cR^\pi(s) \coloneqq
    \cR_{\Theta, f} \rbr{ \tsum_{t=0}^T c_t(S_t, A_t)}
\end{align}
where the data process $\cbr{(S_t,A_t)}_{t=0}^T$ is generated by 
$S_0=s$, $A_t \sim \pi(\cdot|H_t)$, and 
$S_{t+1}=F_t(S_t,A_t,\xi_t)$, with $H_t=\xi_{[t-1]}$ denoting the history up to stage $t$.
In view of the definition of $\cR_{\Theta,f}$, we have
\begin{align}\label{value_given_theta_finite_horizon_soc}
    \cR^\pi(s)
    =
    \min_{\theta \in \Theta} \cbr{
    V_0^{\pi,\theta}(s)
    \coloneqq
    \EE^\pi_P \sbr{
    f_\theta \rbr{\tsum_{t=0}^T c_t(S_t,A_t)}
    \mid S_0=s
    }}.
\end{align}
It is then clear that \eqref{eq_risk_finite_horizon_soc} is equivalent to
\begin{align}\label{finite_horizon_minmin_soc}
    \min_{\theta \in \Theta} \min_{\pi \in \Pi(\cM)} V_0^{\pi,\theta}(s).
\end{align}

As in the MDP setting, the nonlinearity of $f_\theta$ in \eqref{value_given_theta_finite_horizon_soc} prevents the direct application of the standard dynamic programming equations. 
Similar to Section \ref{sec_finite_mdp}, we next introduce an augmented SOC problem that tracks the accumulated total cost through an additional state variable.

\begin{definition}[Finite-horizon augmented SOC]\label{def_aug_soc_finite_horizon}
    For any $\theta \in \Theta$, the augmented SOC $\tilde{\cM}_\theta$ is defined as follows. 
    The state space is $\tilde{\cS}_t = \cS_t \times \cX_t$, where $\cX_t = [0,t]$, and the action space is $\tilde{\cA}_t = \cA_t$.
    The cost function and the transition function are defined by
\begin{align}\label{cost_transition_aug_soc_finite}
\tilde{c}_t^\theta(s,x,a)
&= f_\theta({x+c_t(s,a)})-f_\theta(x), \nonumber \\
(s',x') &= \tilde{F}_t(s,x,a,\xi),
\end{align}
where $\tilde{F}_t: \cS_t \times \cX_t \times \cA_t \times \Xi_t \to \cS_{t+1} \times \cX_{t+1}$ is defined through
\begin{align*}
    s' = F_t(s,a,\xi), 
    ~~~
    x' = x+c_t(s,a).
\end{align*}
\end{definition}

In view of \eqref{cost_transition_aug_soc_finite} in Definition \ref{def_aug_soc_finite_horizon}, the augmented SOC $\tilde{\cM}_\theta$ tracks the accumulated total cost through the augmented state variable $X_t \in \cX_t$ at every stage $t$.
Given this, the set of history-dependent, randomized policies for $\tilde{\cM}_\theta$ is also $\Pi(\cM)$.
For any policy $\pi \in \Pi(\cM)$ and any $\theta \in \Theta$, we define the value function
\begin{align}\label{eq_value_given_pi_aug_soc}
    \tilde{V}^{\pi,\theta}_0(s,x)
    =
    \EE^\pi_P \sbr{
    \tsum_{t=0}^T \tilde{c}_t^\theta(S_t,X_t,A_t)
    \mid S_0=s, X_0=x
    }.
\end{align}

With Definition \ref{def_aug_soc_finite_horizon} in place, we make the following observation by an argument similar to that of Proposition~\ref{prop_reform_fin_mdpite}.

\begin{proposition}\label{prop_reform_finite_soc}
    We have for any $\pi \in \Pi(\cM)$,
    \begin{align*}
        V_0^{\pi,\theta}(s)
        =
        \tilde{V}^{\pi,\theta}_0(s,0) + f_\theta(0),
        \qquad \forall s \in \cS_0.
    \end{align*}
\end{proposition}

In view of Proposition \ref{prop_reform_finite_soc}, one can immediately obtain the following dynamic programming equations for the inner minimization of \eqref{finite_horizon_minmin_soc} through the augmented SOC $\tilde{\cM}_\theta$.

\begin{proposition}\label{thrm_dp_fix_theta_finite_soc}
    Define the optimal cost-to-go function in the augmented SOC $\tilde{\cM}_\theta$ as
    \begin{align*}
        \tilde{V}_t^\theta(s,x)
        =
        \min_{\pi \in \Pi(\cM)}
        \EE^\pi_P
        \sbr{
        \tsum_{h=t}^T \tilde{c}_h^\theta(S_h,X_h,A_h)
        \mid S_t=s, X_t=x
        }.
    \end{align*}
    Then we have
    $\tilde{V}_T^\theta(s,x)=\min_{a \in \cA_T}\tilde{c}_T^\theta(s,x,a)$ for $(s,x)\in \cS_T\times \cX_T$, and
    \begin{align}\label{dp_val_give_theta_finite_soc}
        \tilde{V}_t^\theta(s,x)
        =
        \min_{a \in \cA_t}
        \tilde{c}_t^\theta(s,x,a)
        +
        \EE_{\xi \sim P_t}
        \sbr{
        \tilde{V}_{t+1}^\theta\bigl(\tilde{F}_t(s,x,a,\xi)\bigr)
        },
        \quad \forall (s,x)\in \cS_t \times \cX_t,
    \end{align}
    for any $0 \leq t \leq T-1$.
    In addition, the policy $\pi^\theta$ defined as
    \begin{align}\label{opt_pi_given_theta_finite_soc}
        \pi^\theta_t(s,x)
        \in
        \Argmin_{a \in \cA_t}
        \cbr{
        \tilde{c}_t^\theta(s,x,a)
        +
        \EE_{\xi \sim P_t}
        \sbr{
        \tilde{V}_{t+1}^\theta\bigl(\tilde{F}_t(s,x,a,\xi)\bigr)
        }}
    \end{align}
    is an optimal policy for $\min_{\pi \in \Pi(\cM)} V^{\pi,\theta}_0(s)$.
\end{proposition}

\begin{proof}
    The first part of the claim \eqref{dp_val_give_theta_finite_soc} is a direct consequence of dynamic equations applied to the augmented SOC $\tilde{\cM}_\theta$. 
    The rest of the claim follows from Proposition \ref{prop_reform_finite_soc}.
\end{proof}

Given Proposition \ref{thrm_dp_fix_theta_finite_soc}, we can obtain the dynamic equations for \eqref{eq_risk_finite_horizon_soc} as follows.

\begin{theorem}\label{cor_dp_optimal_value_finite_soc}
Let $\theta^*$ be an optimal solution of \eqref{finite_horizon_minmin_soc}.
Consider $\{\tilde{V}^{\theta^*}_t\}_{t=0}^T$ defined as in \eqref{dp_val_give_theta_finite_soc}
and $\pi^{\theta^*}$ defined in \eqref{opt_pi_given_theta_finite_soc}. 
Then we have
\begin{align*}
    \min_{\pi \in \Pi(\cM)} \cR^\pi(s)
    =
    \tilde{V}^{\theta^*}_0(s,0) + f_{\theta^*}(0),
\end{align*}
and $\pi^{\theta^*}$ is an optimal policy.
\end{theorem}

In view of Theorem \ref{cor_dp_optimal_value_finite_soc}, the optimal policy $\pi^*$ for \eqref{eq_risk_finite_horizon_soc} is non-randomized and depends on the history only through the current state and the accumulated cost.
We now turn our attention to evaluating $\cR^\pi(s)$ for any given policy $\pi$.
Following a similar argument as before, we proceed to establish the dynamic programming equations of $\cR^\pi(s)$ for any $\pi \in \Pi_{\mathrm{MR}}(\tilde{\cM}_\theta)$.

\begin{theorem}\label{thrm_eval_soc_dp_finite}
    For any $\theta \in \Theta$ and $\pi \in \Pi_{\mathrm{MR}}(\tilde{\cM}_\theta)$, define the value function of $\pi$ in $\tilde{\cM}_\theta$ as
    \begin{align*}
        \tilde{V}_t^{\pi,\theta}(s,x)
        =
        \EE^\pi_P
        \sbr{
        \tsum_{h=t}^T \tilde{c}_h^\theta(S_h,X_h,A_h)
        \mid S_t=s, X_t=x
        }.
    \end{align*}
    Then we have
    $\tilde{V}^{\pi,\theta}_T(s,x)
    =
    \EE_{a \sim \pi_T(\cdot|s,x)}
    \sbr{\tilde{c}_T^\theta(s,x,a)}$
    for any $(s,x)\in \cS_T \times \cX_T$, and
    \begin{align}\label{dp_finite_soc_fix_policy}
        \tilde{V}^{\pi,\theta}_t(s,x)
        =
        \EE_{a \sim \pi_t(\cdot|s,x),\, \xi \sim P_t}
        \sbr{
        \tilde{c}_t^\theta(s,x,a)
        +
        \tilde{V}_{t+1}^{\pi,\theta}\bigl(\tilde{F}_t(s,x,a,\xi)\bigr)
        },
        \quad \forall (s,x)\in \cS_t \times \cX_t,
    \end{align}
    for any $0 \leq t \leq T-1$.
    In addition, denote
    \begin{align*}
        \theta_\pi
        \in
        \Argmin_{\theta \in \Theta}
        \cbr{
        \tilde{V}^{\pi,\theta}_0(s,0) + f_\theta(0)
        }.
    \end{align*}
    We have for any $\pi \in \Pi_{\mathrm{MR}}(\tilde{\cM}_\theta)$,
    \begin{align}
        \cR^\pi(s)
        =
        \tilde{V}^{\pi,\theta_\pi}_0(s,0) + f_{\theta_\pi}(0). \label{fin_soc_pol_eval_val_to_risk}
    \end{align}
\end{theorem}

\begin{proof}
    The first part of the claim \eqref{dp_finite_soc_fix_policy} follows from standard dynamic programming principles applied to the augmented SOC. 
    The rest of the claim follows from the definition of $\cR^\pi$, $\theta_\pi$, and Proposition \ref{prop_reform_finite_soc}.
\end{proof}

\subsection{Sample Complexity of Risk-averse Policy Evaluation and Optimization}
We will now turn our attention to \eqref{eq_risk_finite_horizon_soc} when the underlying distributions $\cbr{P_t}$ of transition noises $\cbr{\xi_t}$ are unknown. 
Instead, we assume that one can generate $n$ i.i.d. samples from $P_t$.
Let $\hat{P}_{t,n}$ be the empirical distribution constructed from such samples, 
and write $\hat{P}_n$ as  shorthand for $\{\hat{P}_{t,n}\}_{t=0}^{T-1}$.

\subsubsection{Risk-averse Policy Evaluation}

Denote  
$\hat{\cR}^\pi_n$ as the risk \eqref{eq_risk_finite_horizon_soc} defined by $\hat{P}_n$,
and $\hat{V}^{\pi, \theta}_{t,n}$ the corresponding value function in the augmented SOC $\tilde{\cM}_\theta$.
We are first interested in the number of samples needed for $\hat{\cR}^\pi_n$ to accurately estimate $\cR^\pi$.
To develop the sample complexity result, we make the following assumption on the underlying SOC problem. 
The same assumption has also been made in \cite{shapiro2025risk} for nested risk-averse SOC problems.

\begin{assumption}
\label{ass-lip_soc-1}
(i) For each $0 \leq t \leq T$, the state space $\cS_t$ (resp. action space $\cA_t$) is compact with radius $R$.  (ii)
There is  a positive constant $L$  such that
\begin{eqnarray*}
&&\abs{ c_t(s, a ) - c_t(s', a') }    \leq L \| (s, a) - (s', a') \|,\\
&&  \left\| F_t(s, a,\xi) - F_t(s', a',\xi) \right\|    \leq L \| (s, a) - (s', a') \|,
\end{eqnarray*}
for all $s, s' \in \cS_t$, $a, a' \in    \cA_t$,  $\xi \in   \Xi_t$ and $0 \leq t \leq T$.
\end{assumption}

As in the MDP setting, we will focus on $\pi \in \Pi_{\mathrm{MR}}(\tilde{\cM}_\theta)$ that is $L_\pi$-Lipschitz with respect to the original and augmented state variables, i.e., 
\begin{align}\label{eq_finite_soc_policy_lip}
    \|\pi_t(\cdot|s,x) - \pi_t(\cdot|s',x')\|_1 \leq L_\pi \norm{(s,x) - (s', x')}, \forall s, s' \in \cS_t, x,x' \in \cX_t.
\end{align}

We first establish the Lipschitz continuity of value functions with respect to both the original and the augmented state variables.

\begin{lemma}\label{lemma_fin_soc_pol_eval_val_lip_X}
    For any $0 \leq t \leq T$ and $s \in \cS_t$, the value functions $\tilde V_t^{\pi, \theta}(s, \cdot)$ and $\hat V_{t,n}^{\pi, \theta}(s, \cdot)$ are Lipschitz continuous with modulus
    \begin{align}\label{soc_finite_lipschitz_value_pi_x}
    L^{\pi, \cX}_{t} = (T-t+1)L_C\left(2 + \tfrac{T-t+2}{2}L_\pi\right). 
    \end{align}
    In addition, for any $x \in \cX_t$,  $\tilde V_t^{\pi, \theta}(\cdot, x)$ and $\hat V_{t,n}^{\pi, \theta}(\cdot, x)$ are Lipschitz continuous with modulus
    \begin{align}
        L^{\pi, \cS}_{T} & = L_C L_\pi + L_C L,  \nonumber \\
        L^{\pi, \cS}_{t} & = (T-t+1) L_C L_\pi + (L_C + L^{\pi, \cX}_{t+1} + L^{\pi, \cS}_{t+1})L . \label{soc_finite_lipschitz_value_pi_s}
    \end{align}
\end{lemma}

\begin{proof}
    The proof of \eqref{soc_finite_lipschitz_value_pi_x} follows from similar lines as in Lemma \ref{lemma_fin_mdp_pol_eval_val_lip}.
    We proceed to establish \eqref{soc_finite_lipschitz_value_pi_s} by backward induction.
For $t = T$ and any $x \in \cX_T$, we have
\begin{align*}
    & |\tilde V_T^{\pi, \theta}(s,x) - \tilde V_T^{\pi, \theta}(s',x)| \\
    \leq &|\EE_{a \sim \pi_T(\cdot|s,x)}[\tilde{c}_T^\theta(s',x,a)] - \EE_{a \sim \pi_T(\cdot|s',x)}[\tilde{c}_T^\theta(s',x,a)]| + {|\EE_{a \sim \pi_T(\cdot|s,x)}[\tilde{c}_T^\theta(s,x,a) - \tilde{c}_T^\theta(s',x,a)]}| \\
    \overset{(a)}{\leq} & \sup_{a \in \cA_T}|\tilde{c}_T^\theta(s',x,a)| \cdot \|\pi_T(\cdot|s,x) - \pi_T(\cdot|s',x)\|_1 + L_C L \|s - s'\| \\
    \leq & L_C L_\pi \|s-s'\| + L_C L\|s-s'\|,
\end{align*}
where $(a)$ follows from Hölder’s inequality, Assumption  \ref{assump_basic} and \eqref{eq_finite_soc_policy_lip}.  Suppose the claim holds for $t+1$, then we obtain
\begin{align}
    & |\tilde V_t^{\pi, \theta}(s,x) - \tilde V_t^{\pi, \theta}(s',x)| \nonumber \\
    \leq & |\EE_{a \sim \pi_t(\cdot|s,x), \xi \sim P_t}[\tilde{c}_T^\theta (s',x,a) + \tilde V_{t+1}^{\pi, \theta}( \tilde F_t(s',x,a,\xi))] - \EE_{a \sim \pi_t(\cdot|s',x), \xi \sim P_t}[\tilde{c}_T^\theta (s',x,a) + \tilde V_{t+1}^{\pi, \theta}( \tilde F_t(s',x,a,\xi))]| \nonumber \\
    & + | \EE_{a \sim \pi_t(\cdot|s,x), \xi \sim P_t}[\tilde{c}_T^\theta(s,x,a) - \tilde{c}_T^\theta(s',x,a) + \tilde V_{t+1}^{\pi, \theta}\left(\tilde F_t(s,x,a,\xi)\right) - \tilde V_{t+1}^{\pi, \theta}(\tilde F_t(s',x,a,\xi))]|. \nonumber \\
    \leq & (T-t+1)L_C L_\pi \norm{s-s'} + L_CL \norm{s-s'} + (L_{t+1}^{\pi, \cX} + L_{t+1}^{\pi, \cS})L \norm{s-s'}, \label{fin_soc_pol_eval_big_phi_lip}
\end{align}
where the last inequality follows from Hölder’s inequality, Assumptions \ref{assump_basic} and \ref{ass-lip_soc-1}, and the inductive hypothesis.
\end{proof}
We now turn our attention back to the estimation error $|\hat{\cR}_n^\pi(s)-\cR^\pi(s)|$. In view of \eqref{fin_soc_pol_eval_val_to_risk}, it suffices to control $(\hat{V}_{0,n}^{\pi,\theta}-\tilde{V}_0^{\pi,\theta})(s,0)$ uniformly over $\theta \in \Theta$. To this end, we derive an error decomposition for $(\hat{V}_{t,n}^{\pi,\theta}-\tilde{V}_t^{\pi,\theta})(s,x)$ into martingale difference terms, to which we apply Freedman's inequality.

\begin{lemma} \label{lemma_fin_soc_pol_eval_berns_sum}
    Fix $0 \leq t \leq T-1$, $s \in \cS_t$, $x \in \cX_t$. Then for any $\delta \in (0, 1)$, with probability at least $1-\delta$, we have
    \begin{align}
        \abs{ (\hat{V}_{t,n}^{\pi, \theta} - \tilde{V}_t^{\pi, \theta})(s,x)} \leq \sqrt{\tfrac{2}{n} \hat{W}_{t,n}^{\pi, \theta} (s,x) \log(\tfrac{2}{\delta})} + \tfrac{2 T L_C}{3n}\log(\tfrac{2}{\delta}), \nonumber
    \end{align}
    where
    \begin{align}
        \hat{W}_{t,n}^{\pi, \theta} (s,x) \coloneqq \tsum_{h=t}^{T-1} \EE_{P}^\pi \left[ \left. \Var_{\xi_h \sim P_h} \left(\hat{V}_{h+1,n}^{\pi, \theta}(\tilde{F}_h(S_h, X_h, A_h, \xi_h))\right) \right| (S_t, X_t) = (s,x) \right]. \nonumber
    \end{align}
\end{lemma}
\begin{proof}
    The proof follows from the same argument as in Lemma \ref{lemma_fin_mdp_pol_eval_berns_sum}. 
    The only difference is that we use $\mu_{t,h}^{\pi}(\cdot|s,x)$ to denote the distribution of $(S_h, X_h)$ under policy $\pi$ starting from $(s,x) \in \cS_t \times \cX_t$. With this notation in place, the same martingale decomposition yields the desired claim.
\end{proof}

Lemma \ref{lemma_fin_soc_pol_eval_berns_sum} suggests that controlling $(\hat{V}_{t,n}^{\pi,\theta}-\tilde{V}_t^{\pi,\theta})(s,x)$ reduces to controlling the variance term $\widehat{W}_{t,n}^{\pi,\theta}$. The difficulty is that the conditional variance therein depends on the empirical value function $\hat{V}_{h+1,n}^{\pi,\theta}$. We therefore establish below that its counterpart defined with the true value function $\tilde{V}_{h+1}^{\pi,\theta}$ scales at the order of $O(T^2)$.
\begin{lemma} \label{lemma_fin_soc_pol_eval_bell_tot_var}
    For all $0 \leq t \leq T-1$, $s \in \cS_t$ and $x \in \cX_t$, we have
    \begin{align}
        W_t^{\pi, \theta}(s,x) \coloneqq \tsum_{h=t}^{T-1} \EE_{P}^\pi \left[ \left. \Var_{\xi_h \sim P_h} \left(\tilde{V}_{h+1}^{\pi, \theta}(\tilde{F}_h(S_h, X_h, A_h, \xi_h))\right) \right| (S_t, X_t) = (s,x) \right] \leq (T-t)^2 L_C^2. \nonumber
    \end{align}
\end{lemma}
\begin{proof}
    The claim follows from the same argument as in Lemma \ref{lemma_fin_mdp_pol_eval_bell_tot_var}.
\end{proof}

Combining Lemma \ref{lemma_fin_soc_pol_eval_val_lip_X}, \ref{lemma_fin_soc_pol_eval_berns_sum} and \ref{lemma_fin_soc_pol_eval_bell_tot_var}, we now proceed to establish the sample complexity of the estimator $\hat{\cR}_n^\pi$ for $\cR^\pi$.
\begin{theorem}\label{thrm_sample_eval_soc_finite}
    Suppose $\pi$ satisfies \eqref{eq_finite_soc_policy_lip}. For any $\epsilon > 0$ and $\delta \in (0,1)$, take
    \begin{align}
         n = \cO\left( \tfrac{T^2 L_C^2}{\epsilon^2} \log \left(\tfrac{T^2\left[T^3 L_C L_\pi L^{T+1}/\epsilon\right]^{d_\cS + 1} R^{d_\cS} (L_\Theta R_\Theta/\epsilon)^{d_\Theta}}{ \delta}\right)\right). \nonumber
    \end{align}
    Then with probability at least $1- \delta$, we have
    \begin{align}
        \abs{\hat \cR_n^\pi(s) - \cR^\pi(s) } \leq \epsilon, ~ \forall s \in \cS_0. \nonumber
    \end{align}
\end{theorem}
\begin{proof}
    For each $0 \leq t \leq T$, let $\cN_{\cX_t}$ (resp. $\cN_{\cS_t}$) denote an $\eta$-net for $[0,t]$ (resp. $\cS_t$) where $\eta > 0$ will be specified later. Similar to \eqref{fin_mdp_pol_eval_sup_def} and \eqref{fin_mdp_pol_eval_sup_disc_def}, we define
    \begin{align}
        \cB_n^{\pi, \theta} & \coloneqq \sup \left\{ \abs{(\hat{V}_{t,n}^{\pi, \theta} - \tilde{V}_t^{\pi, \theta})(s,x)} : 0 \leq t \leq T, s \in \cS_t, x \in [0,t] \right\}, \label{fin_soc_pol_eval_sup_def} \\
        \cB_n^{\pi, \theta}(\eta) & \coloneqq \sup \left\{ \abs{(\hat{V}_{t,n}^{\pi, \theta} - \tilde{V}_t^{\pi, \theta})(s,x)} : 0 \leq t \leq T, s \in \cN_{\cS_t}, x \in \cN_{\cX_t} \right\}. \nonumber
    \end{align}
    From Lemma \ref{lemma_fin_soc_pol_eval_val_lip_X}, we have
    \begin{align}
        \cB_n^{\pi, \theta} \leq 2(L_{0}^{\pi,\cX} + L_{0}^{\pi,\cS})\eta + B_n^{\pi, \theta}(\eta). \label{fin_soc_pol_eval_gen_bound}
    \end{align}
    Following an argument similar to that in Theorem \ref{thrm_sample_fin_mdpite_eval}, it suffices to choose $\eta = \tfrac{\epsilon}{8(L_0^{\pi, \cX} + L_0^{\pi, \cS})}$ and
    \begin{align}
        n = \cO\left( \tfrac{T^2 L_C^2 }{\epsilon^2} \log \left(\tfrac{T \abs{\cN_{\cX_T}} \max_{t}\abs{\cN_{\cS_t}}}{\delta}\right)\right) \overset{(a)}{=} \cO\left( \tfrac{T^2 L_C^2}{\epsilon^2} \log \left(\tfrac{T^2\left[L_C L_\pi T^3 L^{T+1}/\epsilon\right]^{d_\cS + 1} R^{d_\cS}}{\delta}\right)\right) \nonumber
    \end{align}
    to ensure that $\cB_n^{\pi, \theta} \leq \tfrac{3 \epsilon}{4}$ holds with probability at least $1-\delta$. Here $(a)$ follows from $\abs{\cN_{\cX_t}} \leq \tfrac{T}{\eta}$ (resp. $\abs{\cN_{\cS_t}} \leq (1 + 2R/\eta)^{d_\cS}$), the definition of $\eta$, and Lemma \ref{lemma_fin_soc_pol_eval_val_lip_X}, which implies  $L_0^{\pi, \cX} = \cO(T^2 L_C L_\pi)$ and $L_0^{\pi, \cS} = \cO(T^3 L_C L_\pi L^{T+1})$. The proof is concluded by applying the same covering argument as in Theorem \ref{thrm_sample_fin_mdpite_eval} to control $\sup_{\theta \in \Theta} \cB_n^{\pi, \theta}$.
\end{proof}

In view of Theorem \ref{thrm_sample_eval_soc_finite}, the sample complexity for risk-averse policy evaluation with a static risk measure \eqref{eq_risk_functional} satisfying Assumptions \ref{assump_basic} and \ref{ass-lip_soc-1} can be bounded by 
$\tilde{\cO}(\frac{T^4 L_C^2 (d_\cS + d_\Theta)}{\epsilon^2} \log(\tfrac{L_\pi L L_\Theta R_\Theta}{\delta \epsilon}))$.
It is interesting to note that the sample complexity for SOCs increases from that of MDPs (c.f., Theorem \ref{thrm_sample_fin_mdpite_eval}) by a factor of $\cO(T)$.
This is due to the Lipschitz constant of the value function scaling exponentially with respect to the number of stages (Lemma \ref{lemma_fin_soc_pol_eval_val_lip_X}).

\subsubsection{Risk-averse Policy Optimization}


We now proceed to establish the sample complexity for policy optimization \eqref{eq_risk_finite_horizon_soc}. Let us denote by $\hat{V}_{t,n}^{\theta}$ the optimal value function of the augmented SOC, defined in the same way as $\tilde{V}_{t}^{\theta}$ in \eqref{dp_val_give_theta_finite_soc}, with $P$ replaced by $\hat{P}_n$ therein. In addition, we use
\begin{align}
    \cR^*(s)=\min_{\pi\in\Pi(\cM)}\cR^\pi(s),
\qquad
\hat{\cR}_n^*(s)=\min_{\pi\in\Pi(\cM)}\hat{\cR}_n^\pi(s) \nonumber
\end{align}
to denote the optimal risk and its empirical estimate. For the rest of our discussion in this section, we are interested in estimating the optimal risk $\cR^*$ via the estimator $\hat{\cR}_n^*$. That is, for any target precision $\epsilon>0$, we seek to characterize a sufficient sample size $n$ to ensure $\sup_{s\in\cS_0}|\hat{\cR}_n^*(s)-\cR^*(s)|\leq \epsilon,$ and to guarantee that $\hat{\pi}_n^* \in \Argmin_{\pi} \hat{\cR}_n^\pi(s)$ is an $\epsilon$-optimal policy for \eqref{eq_risk_finite_horizon_soc}, with high probability.

Before we proceed, it could be worth discussing the basic argument of this section. As in the MDP setting, we have
\begin{align}\label{ineq_policy_opt_soc_basic}
    \tilde{V}_0^{\pi^\theta, \theta}(s,0) - \hat{V}_{0,n}^{\pi^\theta, \theta}(s,0) \overset{(a)}{\leq} \tilde{V}_0^{\theta}(s,0) - \hat{V}_{0,n}^{\theta}(s,0) \overset{(b)}{\leq} \tilde{V}_0^{\hat{\pi}^\theta, \theta}(s,0) - \hat{V}_{0,n}^{\hat{\pi}^\theta, \theta}(s,0) 
\end{align}  
where $\pi^\theta \in \Argmin_{\pi} \tilde{V}_{0}^{\pi,\theta}(s,0)$ and $\hat{\pi}_n^\theta \in \Argmin_{\pi} \hat{V}_{0,n}^{\pi,\theta}(s,0)$. The main difficulty is that neither $\pi^\theta$ nor $\hat{\pi}_n^\theta$ is randomized, and consequently they do not satisfy \eqref{eq_finite_soc_policy_lip}. To address this, we proceed to approximate the optimal policy by a Lipschitz continuous randomized policy. To this end, we first establish the Lipschitz continuity of the optimal value function $\tilde{V}_{t}^{\theta}$ with respect to both the original state variable and the augmented state variable.

\begin{lemma} \label{lemma_fin_soc_pol_opt_lip_X}
    For any $0 \leq t \leq T$ and $s \in \cS_t$, the value functions $\tilde V_t^\theta(s, \cdot)$ and $\hat V_{t,n}^\theta(s, \cdot)$ are Lipschitz continuous with modulus
    \begin{align}\label{soc_finite_lipsthiz_x_opt}
    L^\cX_{t} = 2(T-t+1)L_C.
    \end{align}
    In addition, 
     for any $0 \leq t \leq T$ and $x \in [0,t]$, $\tilde V_t^\theta(\cdot, x)$ and $\hat V_{t,n}^{\theta}(\cdot, x)$ are Lipschitz continuous with modulus
    \begin{align}
    L^\cS_T  = L_C L , ~~~
        L^\cS_{t}  = (L_C + L^\cX_{t+1} + L^\cS_{t+1})L . \label{soc_finite_lipsthiz_s_opt}
    \end{align}

\end{lemma}

\begin{proof}
The proof of \eqref{soc_finite_lipsthiz_x_opt} follows along lines similar to those in Lemma \ref{lemma_fin_mdp_pol_opt_val_lip}.
We proceed to establish \eqref{soc_finite_lipsthiz_s_opt} by backward induction.
 For $t = T$, we have
 \begin{align}
     |\tilde V_T^\theta(s,x) - \tilde V_T^{\theta}(s',x)| = \large| \inf_{a \in \cA_T} \tilde{c}_T^\theta(s,x,a) - \inf_{a \in \cA_T} \tilde{c}_T^\theta(s',x,a)\large| \leq L_C L \|s - s'\|, \nonumber
 \end{align}
    where the last inequality follows from Assumptions \ref{assump_basic} and \ref{ass-lip_soc-1}.  
    Suppose the claim holds at stage $t+1$, then 
    \begin{align*}
        & |\tilde V_t^\theta(s,x) - \tilde V_t^{\theta}(s',x)|\\
        \leq & \sup_{a \in \cA_t} |\tilde{c}_T^\theta(s,x,a) - \tilde{c}_T^\theta(s',x,a) + \EE_{\xi \sim P_t}[ \tilde V_{t+1}^\theta( \tilde F_t(s,x,a,\xi)) - \tilde V_{t+1}^\theta( \tilde F_{t}(s',x,a,\xi))]| \\
        \overset{(a)}{\leq} & L_C L  \|s - s'\| + \sup_{a \in \cA_t}|\EE_{\xi \sim P_t}[ \tilde V_{t+1}^\theta( \tilde F_t(s,x,a,\xi)) - \tilde V_{t+1}^\theta( \tilde F_{t}(s',x,a,\xi))]| \\
        \overset{(b)}{\leq} & \left(L_C + L^\cX_{t+1} + L^\cS_{t+1}\right) L \|s - s'\|,
    \end{align*}
    where $(a)$ follows from Assumption  \ref{assump_basic} and \ref{ass-lip_soc-1}, while $(b)$ follows from the same argument as \eqref{fin_soc_pol_eval_big_phi_lip}.
\end{proof}

With Lemma \ref{lemma_fin_soc_pol_opt_lip_X} in place, we can show that for any $\epsilon>0$, there exists an $\epsilon$-optimal randomized policy $\pi$ satisfying \eqref{eq_finite_soc_policy_lip} with Lipschitz constant $L_\pi=O(\log(1/\epsilon)/\epsilon)$.

\begin{lemma} \label{lemma_close_to_opt_lip_policy_soc}
    For any $\epsilon > 0$, there exists a policy $\pi$ such that 
    $
            \tilde{V}^{\pi, \theta}_t(s,x) -  \tilde{V}^{ \theta}_t(s,x)  \leq \epsilon$ 
    and 
    $
    \norm{\pi_t(\cdot|s,x) -  \pi_t(\cdot|s',x') 
    }_1 
     \leq 
    L_\pi(\epsilon) \norm{(s,x) - (s',x')}, 
    $
where
\begin{align}\label{def_approx_smooth_policy_opt_soc}
    L_\pi(\epsilon) = \frac{8 (T+1)  (T+2) L_C d_{\cA} }{\epsilon}
    \log \rbr{
    \frac{4 \rbr{1 + 8T (L_C + L_0^\cS + L_0^\cX)L / \epsilon } (T+1)^2 L_C }{\epsilon }
    },
\end{align}
and $L_0^\cS$ and $L_0^\cX$ are defined as in Lemma \ref{lemma_fin_soc_pol_opt_lip_X}.
\end{lemma}

 \begin{proof}
 Let $\tilde{Q}^\theta_t$ denote the optimal Q-function of $\tilde{\cM}_\theta$, i.e.,
 \begin{align*}
     \tilde{Q}^\theta_t(s,x,a) = \tilde{c}^\theta_t(s,x,a) + \EE_{\xi \sim P_t} \sbr{\tilde{V}_{t+1}^\theta (F_t(s,a, \xi), x + c_t(s,a)) }.
 \end{align*}
 We proceed to show that $\tilde{Q}^\theta_t(s,x,a)$ is Lipschitz with respect to $a$ for every $(s,x) \in \cS_t \times \cX_t$. 
 Indeed, we have 
 \begin{align}
     & \abs{\tilde{Q}^\theta_t(s,x,a) - \tilde{Q}^\theta_t(s,x,a')} \nonumber \\
      \leq &  \abs{
     \tilde{c}^\theta_t(s,x,a) - \tilde{c}^\theta_t(s,x,a')
     } + \EE_{\xi \sim P_t}
     \sbr{
     \abs{\tilde{V}_{t+1}^\theta (F_t(s,a, \xi), x + c_t(s,a))  - 
     \tilde{V}_{t+1}^\theta (F_t(s,a', \xi), x + c_t(s,a)) } } \nonumber \\
     & ~~~ +   \EE_{\xi \sim P_t} \sbr{ \abs{
     \tilde{V}_{t+1}^\theta (F_t(s,a, \xi), x + c_t(s,a'))  - 
     \tilde{V}_{t+1}^\theta (F_t(s,a', \xi), x + c_t(s,a'))
     }
     }  \nonumber \\
      \leq  & 
     ( L_C   
     + L_{t+1}^{\cS}  
     + L_{t+1}^{\cX} ) L   \norm{a-a'}, \label{lip_q_in_a_soc}
 \end{align}
 where the last inequality follows from Assumption \ref{ass-lip_soc-1} and Lemma \ref{lemma_fin_soc_pol_opt_lip_X}.

 For any $\vartheta > 0$ and any $0 \leq t \leq T$, let $\cA_t^\vartheta$ be an $\vartheta$-net of $\cA_t$.
 Let $\tilde{\cM}^\vartheta_\theta$ be the same as $\tilde{\cM}_\theta$, except that its action space is ${\cA}_t^\vartheta$ instead of $\cA_t$.
 Let us denote $\tilde{V}^{\vartheta,\theta}_t$ the optimal value function of $\tilde{\cM}^\vartheta_\theta$.
 Let $a_t(s,x) \in \Argmin_{a \in \cA_t^\vartheta} \tilde{Q}^\theta_t(s,x,a)$, and consider the policy $\pi^\vartheta$ defined by
 $
 \pi^\vartheta_t (\cdot|s,x) = \delta_{a_t(s,x)}.
 $
 From the performance difference lemma, we have 
 \begin{align*}
     \tilde{V}^{\pi^\vartheta, \theta}_0 -  \tilde{V}^{ \theta}_0
     & = \EE_{P}^{\pi^\vartheta} \sbr{
     \tsum_{t=0}^T \rbr{ \min_{a \in \cA_t^\vartheta} \tilde{Q}_t^\theta (S_t, X_t, a) - \min_{a \in \cA_t} \tilde{Q}_t^\theta (S_t, X_t, a)  }
     } \\ 
     & \leq (T+1) ( L_C   
     + L_0^{\cS}  
     + L_0^{\cX} ) L   \vartheta,
 \end{align*}
 where the last inequality follows from \eqref{lip_q_in_a_soc}.
 Consequently, we must have 
 \begin{align}\label{ineq_soc_finite_action_opt_value}
 \tilde{V}^{\theta}_0 \leq  \tilde{V}^{\vartheta, \theta}_0
 \leq \tilde{V}^{\theta}_0 + (T+1) ( L_C   
     + L_0^{\cS}  
     + L_0^{\cX} ) L   \vartheta.
 \end{align}
 Clearly, $\tilde{\cM}^\vartheta_\theta$ can also be viewed as an MDP with  a finite action space $\cA_t^\vartheta$. 
 Hence we can apply Lemma \ref{lemma_close_to_opt_lip_policy_mdp} and obtain that for any $\epsilon > 0$, there exists a policy $\pi$ such that
 $
   \tilde{V}^{\pi, \theta}_0(s,x) -  \tilde{V}^{\vartheta, \theta}_0(s,x)  \leq \frac{\epsilon}{2} 
 $
 and 
 \begin{align*}
     \norm{
     \pi_t(\cdot|s,x) -  \pi_t(\cdot|s,x') 
     }_1 
     & \leq 
     \frac{8 (T+1)  (T-t+2) L_C}{\epsilon}
     \log \rbr{
     \frac{4 \max_t \abs{\cA_t^\vartheta} (T+1)^2 L_C }{\epsilon }
     } \abs{x - x'}.
 \end{align*}
 Combining the above observations with \eqref{ineq_soc_finite_action_opt_value}, we obtain 
 \begin{align*}
             \tilde{V}^{\pi, \theta}_0(s,x) -  \tilde{V}^{\theta}_0(s,x) & \leq \frac{\epsilon}{2} 
             + (T+1) ( L_C   
     + L_0^{\cS}  
     + L_0^{\cX} ) L   \vartheta.
 \end{align*}
 We can then conclude the proof by taking $\vartheta = \frac{\epsilon}{2(T+1) ( L_C   
 + L_0^{\cS}  
 + L_0^{\cX} ) L} $,
 and noting that
 $
     \abs{\cA^\vartheta_t} \leq (1 + 2R/\vartheta)^{d_{\cA}}
 $.
 \end{proof}

With Lemma \ref{lemma_close_to_opt_lip_policy_soc} established, we can substitute $\pi^\theta$ in \eqref{ineq_policy_opt_soc_basic} with a Lipschitz continuous randomized policy and invoke Theorem \ref{thrm_sample_eval_soc_finite} to control the left-hand side $(a)$ therein. It remains to control the right-hand side of $(b)$ in \eqref{ineq_policy_opt_soc_basic}. Since the policy $\hat{\pi}_n^\theta$ is sample-dependent, Lemma \ref{lemma_fin_soc_pol_eval_berns_sum} cannot be directly applied. Consequently, we proceed to establish its counterpart for potentially sample-dependent policies.

\begin{lemma}\label{lemma_fin_soc_pol_opt_bern_rad}
    Let $\pi$ be a potentially sample-dependent policy satisfying \eqref{eq_finite_soc_policy_lip}. For any $\eta > 0$ and $\delta \in (0,1)$, with probability at least $1-\delta$, we have
    \begin{align}
        \abs{(\hat{V}_{t,n}^{\pi, \theta} - \tilde{V}_t^{\pi, \theta})(s,x)} \leq & \EE_{P}^\pi \left[\left. \tsum_{h=t}^{T-1} b_{h,n}^{\pi, \theta}(S_h, X_h, A_h) \right| (S_t, X_t) = (s,x) \right] \nonumber \\
        & + 12TL_0^{\pi, \cS} \eta + \sqrt{\tfrac{24T^3L_C L_0^{\pi, \cS} \eta }{n}\log(\tfrac{2T^2 (1 + 2R/\eta)^{d_\cS + d_\cA}}{\delta \eta})},\nonumber
    \end{align}
    for any $0 \leq t \leq T-1, s \in \cS_t, x \in [0,t]$, where
    \begin{align}
        b_{t,n}^{\pi, \theta}(s,x,a) := \sqrt{ \tfrac{2}{n} \Var_{\xi \sim P_t} \left(\hat{V}_{t+1,n}^{\pi, \theta}(\tilde{F}_t(s,x,a,\xi))\right) \log(\tfrac{2T^2 (1 + 2R/\eta)^{d_\cS + d_\cA}}{\delta \eta})} + \tfrac{2TL_C}{3n} \log(\tfrac{2T^2(1 + 2R/\eta)^{d_\cS + d_\cA}}{\delta \eta}). \nonumber
    \end{align}
\end{lemma}
\begin{proof}
    For any $0 \leq t \leq T$, let $\cN_{\cX_t}$ (resp. $\cN_{\cS_t}$ and $\cN_{\cA_t}$) denote the $\eta$-net for $[0,t]$ (resp. $\cS_t$ and $\cA_t$). Similar to \eqref{fin_mdp_pol_opt_gen_bound_disc}, applying Bersntein's inequality with the union bound yields
    \begin{align}
        \abs{(\EE_{\xi \sim \hat{P}_{t,n}} - \EE_{\xi \sim P_t})\left[\hat{V}_{t+1,n}^{\pi, \theta}(F_t(s_\eta,a_\eta,\xi), x_\eta + c_t(s_\eta,a_\eta)\right]} \leq b_{t,n}^{\pi, \theta}(s_\eta, x_\eta, a_\eta), \label{fin_soc_pol_opt_gen_bound_disc}
    \end{align}
    for all $0 \leq t \leq T, s_\eta \in \cN_{\cS_t}, x_\eta \in \cN_{\cX_t}$ and $a_\eta \in \cN_{\cA_t}$, with probability at least $1-\delta$. For each $x \in [0,t]$ (resp. $s \in \cS_t$ and $a \in \cA_t$), let $x_\eta$ be its closest point on $\cN_{\cX_t}$ (resp. $s_\eta \in \cN_{\cS_t}$ and $a_\eta \in \cN_{\cA_t}$). Then, following the same lines as in \eqref{ineq_opt_val_err_uniform_raw}, we have
    \begin{align}
        & \abs{(\tilde{V}_{t,n}^{\pi, \theta} - \tilde{V}_t^{\pi, \theta})(s,x)} \nonumber \\
        \overset{(a)}{\leq} & 2(2(L_{t+1}^{\pi, \cX} + L_{t+1}^{\pi, \cS})L + L_{t+1}^{\pi, \cX})\eta + \EE_{a \sim \pi_t(\cdot|s,x)}\left[ b_{t,n}^{\pi, \theta}(s_\eta, x_\eta, a_\eta) + E_{\xi \sim P_t} \left[ \abs{(V_{t+1, n}^{\pi, \theta} - \tilde{V}_{t+1}^{\pi, \theta})(\tilde{F}_t(s_\eta,x_\eta,a_\eta,\xi))}\right]\right] \nonumber \\
        \overset{(b)}{\leq} & 6L_0^{\pi, \cS} \eta + \EE_{a \sim \pi_t(\cdot|s,x)}\left[ b_{t,n}^{\pi, \theta}(s_\eta, x_\eta, a_\eta) + E_{\xi \sim P_t} \left[ \abs{(V_{t+1, n}^{\pi, \theta} - \tilde{V}_{t+1}^{\pi, \theta})(\tilde{F}_t(s_\eta,x_\eta,a_\eta,\xi))}\right]\right], \label{ineq_opt_val_err_uniform_raw_soc}
    \end{align}
    where $(a)$ follows from Lemma \ref{lemma_fin_soc_pol_eval_val_lip_X}, and $(b)$ follows from \eqref{soc_finite_lipschitz_value_pi_s}, which implies $(L_{t+1}^{\pi, \cS} + L_{t+1}^{\pi, \cX})L \leq L_t^{\pi, \cS}$. Note that the transition \eqref{cost_transition_aug_soc_finite} in the augmented SOC $\tilde{\cM}^\theta$ is translation-equivariant in $x$.
    Combining this observation with Assumption \ref{ass-lip_soc-1} and following similar lines as in \eqref{ineq_var_lipschitz_eval}, we obtain
    \begin{align}
        & \abs{\Var_{\xi \sim P_t}\left(\hat{V}_{t+1,n}^{\pi,\theta}(F_t(s,a,\xi), x + c_t(s,a))\right) - \Var_{\xi \sim P_t}\left(\hat{V}_{t+1,n}^{\pi,\theta}(F_t(s',a',\xi), x' + c_t(s',a'))\right)} \nonumber \\
        \leq & 4(T-t) L_C L_t^{\pi, \cS}\left(\|s-s'\| +  \abs{x-x'} + \|a - a'\|\right). \label{ineq_var_lipschitz_eval_soc}
    \end{align}
    With \eqref{ineq_opt_val_err_uniform_raw_soc} and \eqref{ineq_var_lipschitz_eval_soc} in place, the claim follows from the same argument as in Lemma \ref{lemma_fin_mdp_pol_opt_bern_rad}.
\end{proof}


With Lemma \ref{lemma_fin_soc_pol_opt_lip_X}, \ref{lemma_close_to_opt_lip_policy_soc}, and \ref{lemma_fin_soc_pol_opt_bern_rad} in place, we are ready to establish the sample complexity for estimating the optimal risk $\cR^*$ in \eqref{eq_risk_finite_horizon_soc}. We further show that the result naturally leads to a procedure for constructing an $\epsilon$-optimal policy that satisfies the Lipschitz property in \eqref{eq_finite_soc_policy_lip}.

\begin{theorem}\label{thrm_sample_opt_soc_finite}
    For any $\epsilon > 0$ and $\delta \in (0,1)$, take
    \begin{align}
        n = \cO \left(\tfrac{T^3 L_C^2}{\epsilon^2} \log \left(\tfrac{T^2\left[T^6 L_C^3 L^{T+1} d_\cA \log(T^5 L_C^2 L^{T+2}/\epsilon^2)/\epsilon^3\right]^{d_\cS + d_\cA + 1} R^{d_\cS + d_\cA} (L_\Theta R_\Theta/\epsilon)^{d_\Theta}}{\delta}\right)\right). \nonumber
    \end{align}
    Then with probability at least $1-\delta$, we have
    \begin{align*}
        \abs{\hat \cR_n^*(s) - \cR^*(s)} \leq \epsilon, ~ \forall s \in \cS_0. 
    \end{align*}
\end{theorem}
\begin{proof}
    Let us first fix $\theta \in \Theta$. For any $\epsilon > 0$,  similar to \eqref{ineq_opt_err_at_emp_policy} and \eqref{ineq_opt_err_at_true_policy}, from Lemma \ref{lemma_close_to_opt_lip_policy_soc}, there exists an $L_\pi(\epsilon)$-Lipschitz policy $\pi_\epsilon^\theta$ (resp. $\hat{\pi}_\epsilon^\theta$), such that
    \begin{align}
        \hat{V}_{t,n}^{\theta}(s,x) - \tilde{V}_t^{\theta}(s,x) & \geq (\hat{V}_{t,n}^{\hat{\pi}_\epsilon^\theta, \theta}(s,x) - \tfrac{\epsilon}{4}) - \tilde{V}_t^{\hat{\pi}_\epsilon^\theta, \theta}(s,x), \label{fin_soc_pol_opt_close_comp_emp} \\
        \hat{V}_{t,n}^{\theta}(s,x) - \tilde{V}_t^{\theta}(s,x) & \leq \hat{V}_{t,n}^{\pi_\epsilon^\theta, \theta}(s,x)  - (\tilde{V}_t^{\pi_\epsilon^\theta, \theta}(s,x) - \tfrac{\epsilon}{4}). \label{fin_soc_pol_opt_close_comp_real}
    \end{align}
    Since $\pi_\epsilon^\theta$ does not depend on the samples, 
    one can readily invoke Theorem \ref{thrm_sample_eval_soc_finite} to control 
    $\hat{V}_{t,n}^{\pi_\epsilon^\theta, \theta}(s,x)  - \tilde{V}_t^{\pi_\epsilon^\theta, \theta}(s,x)$.
    Consequently, it remains  to control $\hat{V}_{t,n}^{\hat{\pi}_\epsilon^\theta, \theta}(s,x)  - \tilde{V}_t^{\hat{\pi}_\epsilon^\theta, \theta}(s,x)$. 
    Going forward, let us denote $\pi = \hat{\pi}_\epsilon^\theta$.
    For a given $\eta > 0$ to be specified later, from Lemma \ref{lemma_fin_soc_pol_opt_bern_rad} and the same argument leading up to \eqref{fin_mdp_pol_opt_B_final_ineq}, we obtain
    \begin{align}
        \cB_n^{\pi, \theta} \leq 4T^{3/2}L_C\sqrt{\tfrac{1}{n} \log(\tfrac{2}{\delta'})} + \tfrac{4T^2 L_C}{3n} \log(\tfrac{2}{\delta'}) + 24TL_0^{\pi, \cS}(\epsilon) \eta + \sqrt{6 T L_C L_0^{\pi, \cS}(\epsilon) \eta}, \nonumber
    \end{align}
    where $\delta' = \tfrac{\delta \eta}{T^2 (1+2R/\eta)^{d_\cS + d_\cA}}$ and $\cB_n^{\pi, \theta}$ is defined as in \eqref{fin_soc_pol_eval_sup_def}. Thus, in order to ensure $\cB_n^{\pi, \theta} \leq \tfrac{\epsilon}{2}$, it suffices to choose $\eta = \tfrac{\epsilon^2}{96 T L_C L_0^{\pi, \cS}(\epsilon)}$ and
    \begin{align}
        n = \cO \left( \tfrac{T^3 L_C^2}{\epsilon^2} \log \left(\tfrac{1}{\delta'} \right)\right) \overset{(a)}{=} \cO \left( \tfrac{T^3 L_C^2}{\epsilon^2} \log \left(\tfrac{T^2 \left[T^4 L_C L_\pi(\epsilon)L^{T+1}/\epsilon^2\right]^{d_\cS + d_\cA + 1} R^{d_\cS + d_\cA}}{\delta} \right)\right), \label{soc_sample_opt_raw}
    \end{align}
    where in $(a)$ we substitute the definitions of $\delta'$ and $\eta$, and use Lemma \ref{lemma_fin_soc_pol_eval_val_lip_X}, which implies $L_0^{\pi, \cS}(\epsilon) = \cO(T^3 L_C L_\pi(\epsilon) L^{T+1})$. 
    From  Lemma \ref{lemma_fin_soc_pol_opt_lip_X}, we have $L_0^{\cX} = \cO(T L_C)$ and $L_0^{\cS} = \cO(T^2 L_C L^{T+1})$.
    Combining this with the definition of $L_\pi(\epsilon)$ in \eqref{def_approx_smooth_policy_opt_soc}, we obtain from \eqref{soc_sample_opt_raw} that 
    \begin{align}
        n = \cO \left(\tfrac{T^3 L_C^2}{\epsilon^2} \log \left(\tfrac{T^2\left[T^6 L_C^3 L^{T+1} d_\cA \log(T^5 L_C^2 L^{T+2}/\epsilon^2)/\epsilon^3\right]^{d_\cS + d_\cA + 1} R^{d_\cS + d_\cA}}{\delta}\right)\right). \nonumber
    \end{align}
    The desired claim then follows from the same covering argument over  $\theta \in \Theta$  in Theorem \ref{thrm_sample_fin_mdpite_eval}.
\end{proof}

In view of Theorem \ref{thrm_sample_opt_soc_finite}, the sample complexity for estimating the optimal risk $\cR^*$ up to an $\epsilon$-accuracy can be bounded by $\cO(\tfrac{T^5 L_C^2 (d_\cS + d_\cA + d_\Theta)}{\epsilon^2} \log(\tfrac{L L_\Theta R_\Theta}{\delta \epsilon}))$.
Similar to our discussion following Theorem \ref{thrm_sample_eval_soc_finite}, the obtained sample complexity for SOCs increases from that of MDPs (Theorem \ref{thrm_sample_fin_mdpite_opt}) by a factor of $\cO(T)$ due to the Lipschitz constant of the value functions scaling exponentially with respect to the number of stages. 
   In addition to estimating the optimal risk $\cR^*$, Theorem \ref{thrm_sample_opt_soc_finite} also provides a procedure for constructing a $(\tfrac{3\epsilon}{2})$-optimal policy for the policy optimization problem \eqref{eq_risk_finite_horizon_soc}. In particular, using the definition of $\cB_n^{\pi, \theta}$, \eqref{fin_soc_pol_opt_close_comp_emp} and \eqref{fin_soc_pol_opt_close_comp_real}, we obtain
    \begin{align*}
          \tilde{V}_t^{\hat{\pi}_\epsilon^\theta, \theta}(s,x) 
           \leq \tilde{V}_{t}^{\pi_\epsilon^\theta, \theta}(s,x) + \frac{\epsilon}{4} + \cB^{\hat{\pi}_\epsilon^\theta,\theta}_n + \cB^{\pi_\epsilon^\theta,\theta}_n
          \leq \tilde{V}^\theta_t + \frac{\epsilon}{2} + \cB^{\hat{\pi}_\epsilon^\theta,\theta}_n + \cB^{\pi_\epsilon^\theta,\theta}_n
      \leq \tilde{V}^\theta_t + \frac{3 \epsilon}{2},
    \end{align*}
    where the last inequality uses $\cB^{\hat{\pi}_\epsilon^\theta,\theta}_n + \cB^{\pi_\epsilon^\theta,\theta}_n \leq \epsilon$.
    Since the above holds
    uniformly over all $\theta \in \Theta$,
    from Theorem \ref{cor_dp_optimal_value_finite_soc}, 
    we obtain   
    $
    \cR^{\hat{\pi}_\epsilon^{\hat{\theta}^*}} (s) \leq \cR^* (s) + \frac{3\epsilon}{2},
    $
    where 
    $\hat{\theta}^* \in \Argmin_{\theta \in \Theta} \tilde{V}_0^{\hat{\pi}_\theta, \theta}(s,0) + f_\theta(0)$.

%% file: applications.tex

\section{Applications}\label{sec_application}

We now turn our attention to instantiate the  sample complexity results obtained in Sections \ref{sec_risk_mdp} and \ref{sec_soc} to concrete risk functionals $\cR^\pi  \coloneqq  \cR_{\Theta, f}(\cdot)$ of the form \eqref{eq_risk_functional}. 
In particular, we focus on the case where $\cR_{\Theta, f}$ corresponds to the distributionally robust functional defined as in \eqref{eq_phi_dro_f_form}. 
That is,
\begin{align}\label{app_risk_phi_dual}
\cR_{\Theta, f} (X) =  \min_{\lambda \geq 0, \mu \in \RR}
    \cbr{ f_{\lambda, \mu}(X) \coloneqq  \lambda \tau + \mu + 
    \EE_P \sbr{ (\lambda \phi)^* (X - \mu) } } .
\end{align}
As mentioned in Example \ref{example_phi_divergence_dro}, the above problem is the dual of 
\begin{align}\label{app_risk_phi}
 \sup_{\zeta \in \mathfrak{D}} \int_{\Omega} X(\omega) \zeta(\omega) d P(\omega) , ~~~ \mathrm{s.t.}, ~ ~~ \int_{\Omega} \phi(\zeta(\omega)) d P(\omega) \leq \tau,
\end{align}
where the divergence function $\phi$ is closed and convex, satisfying $\phi(1) = 0$ and $\phi(x) = \infty$ for $x < 0$. 
In the context of MDPs and SOCs, 
$\cR^\pi$ denotes the worst-case expectation when the adversary can change the distribution of the state-action data process within a budget specified by the $\phi$-divergence ball defined in \eqref{app_risk_phi}. 
Our ensuing discussion would be to obtain the concrete sample complexity bounds for risk \eqref{app_risk_phi_dual} with different choices of $\phi$-divergences. 

To this end, one needs to verify Assumption \ref{assump_basic} for the risk of interest \eqref{app_risk_phi_dual}.
Nevertheless, it is worth mentioning that in its original form, $\cR_{\Theta, f}$ does not meet the conditions set out by  Assumption \ref{assump_basic}.
First, it is clear that the set $\Theta = \cbr{(\lambda, \mu): \lambda \geq 0, \mu \in \RR}$ is not bounded. 
In addition, the function $f_{\theta} (z) \coloneqq f_{(\lambda, \mu)}(z)$ is typically not Lipschitz in either $\lambda$ or $\mu$. 
Indeed, one can readily see that the Lipschitz property in Assumption \ref{assump_basic} holds only if the corresponding divergence function $\phi$ has a bounded domain. 
To the best of our knowledge, the only commonly used $\phi$-divergence risk that satisfies this property is the conditional value-at-risk.
We illustrate this issue in the concrete context of Kullback-Leibler divergence. 

\begin{example}[Kullback-Leibler divergence]\label{example_kl}
For the Kullback-Leibler divergence, we have $\phi(x) = x\log x - x +1$.
Consequently, $f_{(\lambda, \mu)}$ is given by  \cite[Example 3.6]{shapiro2017distributionally}
\begin{align*}
   f_{(\lambda, \mu)} (z) = \lambda \tau + \mu + \lambda \exp \sbr{(z-\mu)/\lambda} - \lambda,
\end{align*}
which is clearly non-Lipschitz when $\lambda$ is close to zero. 
\end{example}

In view of Example \ref{example_kl}, we now proceed to introduce and utilize some techniques  developed in \cite{li2026sample} and  constrain $(\lambda, \mu)$ in \eqref{app_risk_phi_dual} to a restricted space $\Theta_{\mathrm{tr}}$. 
As a direct consequence, we will show that all the conditions in Assumption \ref{assump_basic} are satisfied over $\Theta_{\mathrm{tr}}$.
In addition, we provide an explicit characterization for the price of this restriction in terms of the estimation accuracy for $\cR_{\Theta, f}$. 
We start with the following basic observation. 

\begin{lemma}[\cite{li2026sample}]\label{lemma_bd_domain_phi}
For any random variable $X \in [-B, B]$, we have 
\begin{align*}
\cR_{\Theta, f} (X) =  \inf
    \cbr{ f_{\lambda, \mu}(X) \coloneqq  \lambda \tau + \mu + 
    \EE_P \sbr{ (\lambda \phi)^* (X - \mu) }: ~ \lambda \in [0, 2B/\tau], \mu \in [-B, B] }.
\end{align*}
\end{lemma}

 Lemma \ref{lemma_bd_domain_phi} suggests that one can take $\Theta = \cbr{(\lambda, \mu): \lambda \in [0, 2B/\tau], \mu \in [-B, B]}$ in $\cR_{\Theta, f}$ defined in \eqref{app_risk_phi_dual} without loss of generality.
 We now turn to address the potential non-Lipschitz property associated with $f_{\theta}$. 
 To this end, we will dedicate our focus on divergence functions $\phi$ that have superlinear growth, i.e., $\lim_{x \to \infty} \phi(x) / x = \infty$.
 We note that such a restriction is necessary. 
 Indeed,  if the superlinear growth does not hold, then \cite{li2026sample} has shown that 
even single-stage distributionally robust stochastic programming problems with $\phi$-divergence become statistically intractable. 
This corresponds to our setting with $T = 1$.

\begin{lemma}\label{lemma_lipschitz_wo_truncation}
Suppose the divergence function $\phi$ satisfies $\lim_{x \to \infty} \phi(x) / x = \infty$. 
Define its growth function $g_\phi(x) = \inf_{x' \geq x} \frac{\phi(x')}{x'}$ for $x \geq 1$,  and its inverse 
\begin{align*}
g_{\phi}^{-1} (y) = \inf \cbr{x: x \geq 1, ~ g_\phi(x) \geq y}.
\end{align*}
For any $\underline{\lambda} > 0$,
define $L_{(\phi, B, \underline{\lambda})} = g_\phi^{-1} ({ \frac{4B}{ \underline{\lambda}}})$. 
We have that 
$
f_{(\lambda, \mu)}(z) 
$
is Lipschitz continuous with respect to $(\lambda, \mu)$  over domain 
$
\Theta_{\mathrm{tr}} \coloneqq \cbr{(\lambda, \mu): \lambda \in [ \underline{\lambda}, 2B/\tau], ~ \mu \in [-B, B]}
$
with  modulus 
\begin{align*}
L_\Theta = L_{(\phi, B, \underline{\lambda})} + 1 + \tfrac{2(L_{(\phi, B, \underline{\lambda})} + 1) B}{\underline{\lambda}} + \tau,
\end{align*}
and $
f_{(\lambda, \mu)}(z) 
$
is Lipschitz continuous  with respect to $z$
with modulus 
$
 L_{(\phi, B, \underline{\lambda})} .
$
In addition, we have 
\begin{align}\label{approx_quality_lambda}
\cR_{\Theta , f}(X) \leq \inf_{(\lambda, \mu) \in \Theta_{\mathrm{tr}} } \EE_P \sbr{f_{(\lambda, \mu)}(X)} \leq \cR_{\Theta , f}(X) + \underline{\lambda} \tau.
\end{align}
\end{lemma}

\begin{proof}
The desired claim follows from a direct application of Lemma 3.5, 3.6, and 3.7 in \cite{li2026sample}.
\end{proof}

In view of Lemma \ref{lemma_lipschitz_wo_truncation}, by restricting $(\lambda, \mu)$ to the truncated domain $\Theta_{\mathrm{tr}}$, the corresponding 
$f_{(\lambda, \mu)}$ and $\Theta_{\mathrm{tr}}$ satisfy Assumption \ref{assump_basic}.
The resulting risk functional $\cR_{\Theta_{\mathrm{tr}} , f}$, given \eqref{approx_quality_lambda}, 
closely tracks that of the true risk $\cR_{\Theta, f}$, provided the restriction $\underline{\lambda}$ in $\Theta_{\mathrm{tr}}$ is small enough. 
We are now ready to establish sample complexity results for risk-averse MDP and SOC problems with general $\phi$-divergences.

\subsection{Risk-averse MDPs with $\phi$-divergences}\label{subsec_application_mdp}

Using Lemma \ref{lemma_lipschitz_wo_truncation}, we first establish the general sample complexity for policy evaluation and optimization in risk-averse MDPs. 

\begin{theorem} \label{thrm_samp_comp_mdp_phi}
    Suppose $\pi$ satisfies \eqref{eq_finite_mdp_policy_lip}. For any $\epsilon > 0$ and $\delta \in (0,1)$, take
 \begin{align}\label{eval_mdp_application}
 n = \cO \left(\tfrac{T^2 g_\phi^{-1}(16T\tau/\epsilon)^2}{\epsilon^2} \log \left(\tfrac{T |\cS| L_\pi \tau g_\phi^{-1}(16T \tau /\epsilon)}{\delta \epsilon} \right)\right).
 \end{align}
 Then with probability at least $1-\delta$, we have $|\hat{\cR}_n^\pi(s) - \cR^{\pi}(s)| \leq \epsilon$ for all $s \in \cS_0.$
    In addition, take
   \begin{align}\label{opt_mdp_application}
   n = \cO \left(\tfrac{T^3 g_\phi^{-1}(16T \tau /\epsilon)^2}{\epsilon^2} \log \left(\tfrac{T |\cS| |\cA| \tau g_\phi^{-1}(16T \tau /\epsilon)}{\delta \epsilon} \right)\right).
   \end{align}
    Then with probability at least $1-\delta$, we have  $| \hat{\cR}^*_n(s)  -\cR^*(s) | \leq \epsilon $ for all $ s \in \cS_0$.
\end{theorem}
\begin{proof}
    Let $\underline{\lambda} > 0$ and $L_\phi = g_\phi^{-1}(\tfrac{4T}{\underline{\lambda}})$. 
    In addition,  define  $\cR^\pi  \coloneqq \cR_{\Theta, f} (\cdot)$ and $\cR^{\pi, \mathrm{tr}}  \coloneqq \cR_{\Theta_{\mathrm{tr}}, f} (\cdot)$.
Then from Lemma \ref{lemma_lipschitz_wo_truncation}, we obtain
    \begin{align*}
        \left|\hat{\cR}_n^\pi(s) - \cR^{\pi}(s)\right| \leq 2\underline{\lambda} \tau +  \left|\hat{{\cR}}_n^{\pi, \mathrm{tr}}(s) - {\cR}^{\pi, \mathrm{tr}}(s)\right|.
    \end{align*}
   Then \eqref{eval_mdp_application} follows by choosing $\underline{\lambda} = \tfrac{\epsilon}{4 \tau}$ in Lemma \ref{lemma_lipschitz_wo_truncation}, followed by  Theorem \ref{thrm_sample_fin_mdpite_eval}    applied 
   to $\cR^{\pi, \mathrm{tr}}$.  
    %
    In addition, \eqref{opt_mdp_application} follows from the same argument, except that we use  Theorem \ref{thrm_sample_fin_mdpite_opt} in place of Theorem \ref{thrm_sample_fin_mdpite_eval}.
\end{proof}

Clearly, Theorem \ref{thrm_samp_comp_mdp_phi} suggests that the sample complexity of risk-averse MDPs with a $\phi$-divergence-based risk functional depends on the growth of  $\phi$. 
It is worth noting here that such a dependence is indeed necessary and cannot be improved in general.
Indeed, for $T=1$ our considered MDP can be viewed as a single-stage distributionally robust optimization, and \cite{li2026sample} has recently established an information-theoretic lower bound that depends on the growth function of $\phi$. 

\begin{corollary}[Conditional value-at-risk] \label{cor_samp_comp_mdp_cvar}
    Suppose $\pi$ satisfies \eqref{eq_finite_mdp_policy_lip}. For any $\epsilon > 0$ and $\delta \in (0,1)$, take
    \begin{align*}
    n = \cO \left(\tfrac{T^2}{\alpha^2\epsilon^2} \log \left(\tfrac{T |\cS| L_\pi}{\delta \alpha \epsilon} \right)\right), 
    \end{align*}
 Then with probability at least $1-\delta$, we have $|\hat{\cR}_n^\pi(s) - \cR^{\pi}(s)| \leq \epsilon$ for all $s \in \cS_0.$
    In addition, take
    \begin{align*}
    n = \cO \left(\tfrac{T^3}{\alpha^2 \epsilon^2} \log \left(\tfrac{T |\cS| |\cA|}{\delta \alpha \epsilon} \right)\right).
    \end{align*}
    Then with probability at least $1-\delta$, we have  $| \hat{\cR}^*_n(s)  -\cR^*(s) | \leq \epsilon $ for all $ s \in \cS_0$.
\end{corollary}
\begin{proof}
    The proof follows directly from Theorem \ref{thrm_sample_fin_mdpite_eval}, \ref{thrm_sample_fin_mdpite_opt}, and \ref{thrm_samp_comp_mdp_phi}, by noting that for conditional value-at-risk, we have \cite{shapiro2017distributionally}
    \begin{align*}
    \cR^\pi \rbr{ s} = \min_{\mu} \cbr{ \mu + \frac{1}{\alpha} \EE^\pi_{\PP} [{ \tsum_{t=0}^T c(S_t, A_t) - \mu}]_+ }.
    \end{align*}
   Consequently, one can take  
    $L_C = L_\Theta = 1/\alpha$ and $\Theta = [0,T]$.
\end{proof}

It should be noted that conditional value-at-risk has received considerable attention in the recent development of risk-averse MDPs \cite{bastani2022regret,wang2023near,ni2024risk}. 
Below, we proceed to instantiate Theorem \ref{thrm_samp_comp_mdp_phi} and obtain some new sample complexity results for commonly used $\phi$-divergences, 
including KL-divergence, $\chi^2$-divergence, and more generally 
the Cressie-Read class of divergences.

\begin{corollary}[Kullback-Leibler divergence] \label{cor_samp_comp_mdp_kl}
Let $\phi(x) = x\log x - x + 1$ in \eqref{app_risk_phi_dual}.
    Suppose $\pi$ satisfies \eqref{eq_finite_mdp_policy_lip}. For any $\epsilon > 0$ and $\delta \in (0,1)$, take
    \begin{align*}
    n = \cO \left(\tfrac{T^3 \tau \exp(T\tau/\epsilon)^2}{\epsilon^3} \log \left(\tfrac{T |\cS| L_\pi \tau}{\delta \epsilon} \right)\right).
    \end{align*}
 Then with probability at least $1-\delta$, we have $|\hat{\cR}_n^\pi(s) - \cR^{\pi}(s)| \leq \epsilon$ for all $s \in \cS_0.$
    In addition, take
    \begin{align*}
    n = \cO \left(\tfrac{T^4 \tau \exp(T\tau/\epsilon)^2}{\epsilon^3} \log \left(\tfrac{T |\cS| |\cA| \tau}{\delta \epsilon} \right)\right), 
    \end{align*}
    Then with probability at least $1-\delta$, we have  $| \hat{\cR}^*_n(s)  -\cR^*(s) | \leq \epsilon $ for all $ s \in \cS_0$.
\end{corollary}
\begin{proof}
    In view of Theorem \ref{thrm_samp_comp_mdp_phi}, it only remains to control $g_\phi^{-1}(\frac{16T \tau}{\epsilon})$. A direct calculation yields
    $$g_\phi(x) = \inf_{x' \geq x} \tfrac{x' \log x' - x' + 1}{x'} = \log x - 1 + \tfrac{1}{x}.$$
    Consequently, the definition of $g_\phi^{-1}$ implies $g_\phi^{-1}(\frac{16T \tau }{\epsilon}) \leq e^{1 + 16T\tau/\epsilon}$.
\end{proof}

Corollary \ref{cor_samp_comp_mdp_kl} appears to be the first sample complexity result for static risk-averse MDPs with KL-divergence. 
In particular, it can be readily seen that the obtained sample complexity scales exponentially with respect to the horizon $T$ and target precision $\epsilon$. 
It should be noted that the dependence on $\epsilon$ is in general not improvable and hence worst-case optimal even for $T=1$ (c.f. \cite{li2026sample}). 
It remains interesting to study whether the dependence on~$T$ can be further improved.
As our last result, we establish the sample complexity for the Cressie-Read family of divergence functions indexed by $k > 1$. 

\begin{corollary}[Cressie-Read family]\label{corr_read}
    \label{cor_samp_comp_mdp_cr}
    For any $k > 1$, let $\phi(x) = \tfrac{x^k -kx + k-1}{k(k-1)}$. 
    Suppose $\pi$ satisfies \eqref{eq_finite_mdp_policy_lip}. For any $\epsilon > 0$ and $\delta \in (0,1)$, take
    \begin{align*}
    n = \cO \left(\tfrac{T^2 (k^2T \tau/\epsilon + k)^{2/(k-1)}}{\epsilon^2} \log \left(\tfrac{T |\cS| L_\pi \tau (k^2T \tau/\epsilon + k)^{1/(k-1)}}{\delta \epsilon} \right)\right). 
    \end{align*}
 Then with probability at least $1-\delta$, we have $|\hat{\cR}_n^\pi(s) - \cR^{\pi}(s)| \leq \epsilon$ for all $s \in \cS_0.$
    In addition, take
    \begin{align*}
    n = \cO \left(\tfrac{T^3 (k^2T \tau/\epsilon + k)^{2/(k-1)}}{\epsilon^2} \log \left(\tfrac{T |\cS| |\cA| \tau (k^2T \tau/\epsilon + k)^{1/(k-1)}}{\delta \epsilon} \right)\right).
    \end{align*}
    Then with probability at least $1-\delta$, we have  $| \hat{\cR}^*_n(s)  -\cR^*(s) | \leq \epsilon $ for all $ s \in \cS_0$.
\end{corollary}
\begin{proof}
    Similarly to Corollary \ref{cor_samp_comp_mdp_kl}, it only remains to bound $g_\phi^{-1}(\frac{16T \tau}{\epsilon})$. For any $x \geq 1$, we have
    $$g_\phi(x) = \inf_{x' \geq x} \tfrac{(x')^{k-1} -k + (k-1)/x'}{k(k-1)} = \tfrac{x^{k-1} -k + (k-1)/x}{k(k-1)}.$$
    Consequently, we obtain $g_\phi^{-1}(16T\tau/\epsilon) \leq (16k(k-1)T \tau/\epsilon + k)^{1/(k-1)}.$
    The desired claim then follows from applying Theorem \ref{thrm_samp_comp_mdp_phi}.
\end{proof}

It is worth mentioning here that $\chi^2$-divergence corresponds to the special case of $k = 2$.
In addition, Corollary \ref{corr_read} seems to be the first sample complexity result for risk-averse MDPs with static risk functionals generated by the Cressie-Read class of $\phi$-divergences.

\subsection{Risk-averse SOCs with $\phi$-divergences}

We now turn our attention to specialize the obtained sample complexity results in Section \ref{sec_soc} to risk-averse stochastic optimal control, when the considered risk functional $\cR^\pi \coloneqq \cR_{\Theta, f}(\cdot)$ is given by \eqref{app_risk_phi_dual}.
Similarly to Theorem \ref{thrm_samp_comp_mdp_phi}, we first establish the sample-complexity with general divergence functions.

\begin{theorem} \label{thrm_samp_comp_soc_phi}
    Suppose $\pi$ satisfies \eqref{eq_finite_soc_policy_lip}. For any $\epsilon > 0$ and $\delta \in (0,1)$, take
    \begin{align*}
    n = \cO \left(\tfrac{T^3 g_\phi^{-1}(16T\tau/\epsilon)^2 d_\cS}{\epsilon^2} \log \left(\tfrac{T L_\pi L \tau g_\phi^{-1}(16T \tau /\epsilon)}{\delta \epsilon} \right)\right).
    \end{align*}
 Then with probability at least $1-\delta$, we have $|\hat{\cR}_n^\pi(s) - \cR^{\pi}(s)| \leq \epsilon$ for all $s \in \cS_0.$
    In addition, take
    \begin{align*}
    n = \cO \left(\tfrac{T^4 g_\phi^{-1}(16T \tau /\epsilon)^2 (d_\cS + d_\cA)}{\epsilon^2} \log \left(\tfrac{T L \tau g_\phi^{-1}(16T \tau /\epsilon)}{\delta \epsilon} \right)\right).
    \end{align*}
    Then with probability at least $1-\delta$, we have  $| \hat{\cR}^*_n(s)  -\cR^*(s) | \leq \epsilon $ for all $ s \in \cS_0$.
\end{theorem}

The specialization of Theorem \ref{thrm_samp_comp_soc_phi} to different concrete choices of $\phi$ largely follows the same computation as presented for MDPs in Section \ref{subsec_application_mdp}. 
Therefore  we omit their details for presentation simplicity. 
While sample complexity for risk-averse SOCs with nested risk measures has been recently investigated in \cite{shapiro2025risk}, Theorem \ref{thrm_samp_comp_soc_phi} seems to be the first result for risk-averse SOCs with static risk measures.

%% file: conclusion.tex

\section{Concluding Remarks}
To conclude our discussions in this manuscript, it is worth noting that the dynamic equations for risk-averse MDP (SOC) with static risk measures are by no means unique \cite{bauerle2014more, muni2026reward,bauerle2011markov}. 
A potential consequence of this is that in defining the notion of time-consistency \cite{shapiro2009time}, the corresponding tail-problems for which the policy is optimal with respect to depend on the dynamic equations one considers. 
In addition, while we mainly focus on dynamic programming equations and sample complexities in this manuscript, it could be also rewarding to develop  efficient solution methods for risk-averse MDPs and SOCs problems with static risk measures via the proposed DP equations.